\documentclass[11pt]{amsart}

\usepackage{amscd,amssymb}
\usepackage[arrow,matrix]{xy}
\usepackage{xcolor}
\usepackage[colorlinks=true,backref,linkcolor=blue,citecolor=purple]{hyperref}
\usepackage[utf8]{inputenc}
\usepackage{color,xcolor}
\usepackage{pgf,tikz}
\usetikzlibrary{arrows}
\input xy
\xyoption{all}
\usepackage[colorinlistoftodos]{todonotes}
\usepackage{comment}
\usepackage{bm}

\theoremstyle{plain}
\newtheorem{thm}[subsection]{Theorem}
\newtheorem{lem}[subsection]{Lemma}
\newtheorem{prop}[subsection]{Proposition}
\newtheorem{cor}[subsection]{Corollary}

\theoremstyle{definition}
\newtheorem{rem}[subsection]{Remark}
\newtheorem{defn}[subsection]{Definition}
\newtheorem{exmp}[subsection]{Example}

\newtheorem{notn}[subsection]{Notation}
\newtheorem{headtheorem}{Theorem}

\numberwithin{equation}{section}
\def\deg{{\mathrm{deg}}}
\def\ann{{\mathrm{ann}}}
\def\Fitt{{\mathrm{Fitt}}}
\def\KK{{\mathbb{K}}}

\def\PP{{\mathbb{P}}}
\def\ZZ{{\mathbb{Z}}}
\def\HH{{\mathbb{H}}}
\def\NN{{\mathbb{N}}}
\def\N{{\mathbb{N}}}
\def\MM{{\mathbb{M}}}

\def\Spec{{\mathrm{Spec}}}
\def\Sylv{{\mathrm{Sylv}}}
\def\Syz{{\mathrm{Syz}}}
\def\sylv{{\mathrm{sylv}}}
\def\HF{{\mathrm{HF}}}
\def\HP{{\mathrm{HP}}}

\def\Hom{{\mathrm{Hom}}}
\def\mm{{\mathfrak{m}}}
\def\bb{{\mathfrak{b}}}

\def\bx{{\bm{x}}}
\def\bff{{\bm{f}}}

\def\bU{{\bm{U}}}
\def\bF{{\bm{F}}}

\def\bd{{\bm{d}}}
\def\bnu{{\bm{\nu}}}
\def\bmu{{\bm{\mu}}}
\def\balpha{{\bm{\alpha}}}
\def\bgamma{{\bm{\gamma}}}
\def\bbeta{{\bm{\beta}}}
\def\bdelta{{\bm{\delta}}}
\def\b0{{\bm{0}}}
\DeclareMathOperator{\Supp}{ Supp}

\DeclareMathOperator{\sat}{ sat}

\def\Dc{{\mathcal{D}}}
\def\Jc{{\mathcal{J}}}

\def\Af{{\mathfrak{A}}}

\newcommand{\red}[1]{\textcolor{red}{#1}}

\newcommand{\yellow}[1]{\textcolor{yellow}{#1}}
\newcommand{\green}[1]{\textcolor{green}{#1}}

\begin{document}

\title{Multigraded Sylvester forms, Duality and Elimination Matrices}

\author{Laurent Busé}
\address{Université Côte d'Azur, Inria, 2004 route des Lucioles, 06902 Sophia Antipolis, France.}
\email{laurent.buse@inria.fr}

\author{Marc Chardin}
\address{Institut de Mathématiques de Jussieu, CNRS \& Sorbonne Université, France}
\email{marc.chardin@imj-prg.fr}

\author{Navid Nemati}
\address{Université Côte d'Azur, Inria, 2004 route des Lucioles, 06902 Sophia Antipolis, France.}
\email{navid.nemati@inria.fr}


\begin{abstract}
In this paper we study the equations of the elimination ideal associated with $n+1$ generic multihomogeneous polynomials defined over a product of projective spaces of dimension $n$. We first prove a duality property and then make this duality explicit by introducing multigraded Sylvester forms. These results provide a partial generalization of similar properties that are known in the setting of homogeneous polynomial systems defined over a single projective space. As an important consequence, we derive a new family of elimination matrices that can be used for solving zero-dimensional multiprojective polynomial systems by means of linear algebra methods.  
\end{abstract}

\maketitle

\section{Introduction}\label{sec:introduction}

The elimination of several variables from a system of homogeneous polynomial equations is a fundamental operation that is involved in many computational methods in algebraic geometry. It has received a lot of interest in the existing literature, from the very beginning of elimination theory with contributions by Cayley, Sylvester, Macaulay, and many others, to its more recent developments deeply rooted in modern algebraic geometry (e.g.~\cite{GKZ}). The geometric interpretation of elimination is in terms of projection maps, as illustrated by the famous Main Theorem of elimination theory \cite[Chapter V]{EH00}. On the algebraic side, the elimination of variables from homogeneous equations relies on the saturation of the ideal generated by these equations with respect to the ideal generated by the variables. The importance of this saturated ideal was already noticed from the classical literature on elimination theory; polynomials in this ideal are called \emph{inertia forms} after Hurwitz. The case of generic complete intersections has been extensively studied and structural results are known, especially in the case of $n+1$ generic homogeneous polynomials in $n+1$ variables, as it corresponds to the setting of the Macaulay resultant, which is of particular interest. To be more precise, let us introduce some notation.

Let $F_0,\ldots,F_n$ be the $n+1$ homogeneous polynomials in the variables $x_0,\ldots,x_n$ with indeterminate coefficients. We denote by $d_i$ the degree of $F_i$ which is a homogeneous polynomial in the graded polynomial ring $C=A[x_0,\ldots,x_n]$, where $A$ stands for the universal ring of coefficients, i.e.~the polynomial ring over the integers whose variables are all the coefficients of $F_0,\ldots,F_n$ (see Section \ref{subsec:MacMat} for more details). The saturation of the ideal $I=(F_0,\ldots,F_n)$ with respect to the ideal $\mm=(x_0,\ldots,x_n)$ is the ideal $I^{\sat}:=(I:\mm^\infty)$. It is well known that the graded component of degree 0 of $I^{\sat}$, $(I^{\sat})_0=I^{\sat}\cap A$, is a principal and prime ideal of $A$ generated by the Macaulay resultant of $F_0,\ldots,F_n$. Setting $\delta:=\sum_{i=0}^{n}(d_i-1)$, we also have that $I^{\sat}_\nu=I_{\nu}$ for all $\nu>\delta$, a property that is often summarized by saying that all inertia forms of degree $\nu>\delta$ are trivial inertia forms, as they can be obtained as polynomial combinations of the $F_i$'s. Actually, there exists a duality that provides many structural properties of the quotient $I^{\sat}/I$. This duality can be made explicit by means of Bezoutian determinants, also called Morley forms; see \cite{J91,J97,SS01} (where this duality is actually proved in the more general context of anisotropic graded rings). In particular, one gets that the graded components $(I^{\sat}/I)_\nu$ are free $A$-modules for all $\nu>\delta-\min_i d_i$ and explicit bases are known. More precisely, the graded component of degree $\delta$ is an $A$-module of rank one which is generated by the (twisted) Jacobian determinant of the polynomials $F_0,\ldots,F_n$. Basis for the other graded components are obtained by means of Sylvester forms associated with $F_0,\ldots,F_n$, as proved in \cite{J97} (see Section \ref{subsec:SGSylvForm} for a review of this construction). For lower degree $\nu$, to the best of our knowledge, their is no known explicit sets of generators in general; except for very particular settings as the case $n=1$ \cite{J97,Bus09}.

An important application of the above results is solving zero-dimensional polynomial systems, with coefficients in a field $\KK$, by means of linear algebra techniques. The main ingredient is to build matrices whose columns are filled with the coefficients of some inertia forms of a given degree $\nu$ with respect to a given basis of homogeneous polynomials of degree $\nu$ (e.g.~the canonical monomial basis). Typical examples, in the generic setting, are matrices of the maps (of free $A$-modules)
\begin{eqnarray}\label{eq:FirstKoszulMap}
	\oplus_{i=0}^n C_{\nu-d_i} & \xrightarrow{(F_0,\ldots,F_n)} & C_\nu \\ \nonumber
	(G_0,\ldots,G_n) & \mapsto & \sum_{i=0}^n G_iF_i,
\end{eqnarray}
which correspond to the graded components of the presentation matrix of $I$. The columns of these matrices are filled with the coefficients of trivial inertia forms of degree $\nu$. When the coefficients of the $F_i$'s are specialized to a field $\KK$, which corresponds to a polynomial system over $\PP^n_\KK$ that we denote by $f_0=\cdots=f_n=0$, it turns out that for all $\nu>\delta$ the matrices \eqref{eq:FirstKoszulMap} are not surjective after specialization if and only if the $f_i$'s have a common root in $\PP^n_{\overline{\KK}}$, $\overline{\KK}$ denoting the algebraic closure of $\KK$. More importantly, if the number of common roots of the $f_i$'s is finite, then the corank of these matrices is precisely this number; we will quote this property as the \emph{drop-of-rank property}. This was first noticed and proved by Lazard in his foundational paper \cite{Laz}. He also showed that linear algebra techniques, such as Gaussian elimination, can be performed on the cokernel of these matrices to extract (approximate values of) the roots of the polynomial system considered. Another close approach relies on the computation of eigenvalues and eigenvectors from the cokernel of these matrices. It is based on the famous Eigenvalue Theorem that goes back to Stickelberger (see \cite{Cox20} for more details). In this paper we will focus on the construction of elimination matrices and we refer the readers to \cite{CLO98} and \cite{TelenThesis} for more details on the solving of polynomial systems from elimination matrices having the drop-of-rank property. We retain that from a computational efficiency perspective it is  useful to build as small as possible such matrices. Sylvester forms can be used for that purpose, by completing the matrices \eqref{eq:FirstKoszulMap} in degree $\nu$ such that $\delta-\min_i d_i\leq \nu<\delta$, so that the drop-of-rank property is preserved. We will review this construction and its related results with more details in Section \ref{sec:SGsection}. 

\medskip

The main objective of this paper is the extension of the above results from the single graded case to the multigraded case. Thus, instead of considering homogeneous polynomials in a single set of variables, we  consider multihomogeneous polynomials. In  geometric terms, such multihomogeneous equations define hypersurfaces in a product of projective spaces $\PP^{n_1}\times \ldots\times \PP^{n_r}$. We will consider multiprojective polynomial systems of $n+1$ hypersurfaces where $n=\sum_{i=1}^r n_i$. Let $F_0,\dots, F_n \in C $ be the $n+1$ generic multihomogeneous polynomials of degrees $\bd_0,\dots ,\bd_n$, where $C = A[\bx_1,\dots, \bx_r]$ with $A$ the universal ring of coefficients.    Such systems arise in many contexts and applications and there is a rich literature on their study. In particular, the theory of multiprojective resultant has been extensively covered, for instance by Gelfand, Kapranov and Zelevinsky in \cite{GKZ} from a geometric perspective, and by Rémond in \cite{Remond} and Jouanolou and his students in \cite{ChThesis88,ChaichaaThesis} from more algebraic and computational points of view. However, to the best of our knowledge, structural results on the corresponding saturated ideal are not available, as well as compact elimination matrices built from non trivial inertia forms and having the drop-of-rank property. The main contribution of this paper is to provide new results in this direction.

 In the multigraded setting, the saturation of ideals is done with respect to the product $\bb=\prod_{i=1}^r \mm_i$ where $\mm_i$ is the irrelevant ideal associated with the projective space $\PP^{n_i}$. Similarly to the  single graded case, one can define a critical degree $\bdelta:=\sum_{i=0}^n \bd_i- (n_1+1,\dots, n_r+1) $.
 In Section \ref{sec:duality}, we prove a duality property for the  ideal $I$ generated by $n+1$ generic multihomogeneous polynomials (Theorem \ref{thm: duality}) and derive some consequences:

\begin{headtheorem}
Let $F_0,\dots, F_n$ be the $n+1$ generic multihomogeneous polynomials. There exists $\Theta \subset \bdelta -\NN^r$, such that
$$
(I^{\sat}/I)_\bmu = \Hom_{A}((C/I)_{\bdelta-\bmu},A)
$$
for every $\bmu\in \Theta$.
\end{headtheorem}

\noindent The region $\Theta$ is explicitly described in Theorem \ref{thm: duality}. In particular, we recover the known fact that $(I^{\sat}/I)_{\bdelta}$ is a free $A$-module of rank $1$ (Corollary \ref{cor- main theorem}) and more generally we show that there is a region of degrees for which graded components of the quotient $I^{\sat}/I$ are free $A$-modules.
In Section \ref{sec:MultiGradedSylvesterForms} we provide some explicit basis for these graded components. The graded component of degree $\bdelta$ is generated by a multigraded twisted Jacobian; this result already appeared in the existing literature, for instance in \cite{CCD97} and \cite{CDS98} (in the more general context of toric geometry) and in \cite{ChThesis88}. After reviewing this construction in Section \ref{sec:JouanoloySylvForm}, we introduce multigraded Sylvester forms (Definition \ref{def:MGSylvForm}) and prove that they provide the expected basis for some multigraded components of $I^{\sat}/I$ (Theorem \ref{thm:main}):
\begin{headtheorem}
Let $F_0,\dots, F_n$ be the $n+1$ generic multihomogeneous polynomials of degrees $\bd_i:= (d_{i,1},\dots, d_{i,r})$ for $0\leq i\leq n$.  For every $\bmu:= 
(\mu_1,\dots, \mu_r)$ such that $0\leq \mu_j < \min_i d_{i,j}
$, the set of 
multigraded Sylvester forms yields an $A$-basis of the free 
$A$-module $(I^{\sat}/I)_{\bdelta- \bmu}$.
\end{headtheorem}
\noindent From all this, in Section \ref{sec:MGelimMat} we construct new elimination matrices that are built using non trivial inertia forms and having the drop-of-rank property. We conclude with some illustrative examples, including some matrices due to Dixon \cite{Dixon} that we recover and that are associated with the resultant of three generic bihomogeneous polynomials of the same bidegree.

\section{Preliminaries: the single graded case}\label{sec:SGsection}

In this section we review some results on elimination matrices associated with zero-dimensional polynomial systems defined by $n+1$ equations over a single projective $\PP^n_\KK$, where $\KK$ is field. More precisely, we consider $n+1$ homogeneous polynomials $f_0,\ldots,f_n$ in the polynomial ring $R=\KK[x_0,\ldots,x_n]$ of positive degrees $d_0,\ldots,d_n$. We describe a family of matrices that only depend on the coefficients of the $f_i$'s and that can be used to solve this polynomial system by means of linear algebra techniques. 

Our strategy to study elimination ideals of homogeneous polynomial systems has two steps. We first  consider the generic setting, i.e.~the coefficients of the $f_i$'s are seen as variables. Subsequently,  we proceed by analyzing specializations of these coefficients to a field $\KK$. We will denote by $\bar{\KK}$ the algebraic closure of $\KK$. 

\subsection{Macaulay-type elimination matrices}\label{subsec:MacMat}  We consider the generic polynomial system of $n+1$ homogeneous polynomials in $n+1$ variables of positive degrees $d_0,d_1,\ldots,d_n$, over a commutative ring $k$. Thus, for all $i=0,\ldots,n$ we will denote by 
$$F_i(x_0,\ldots,x_n)=\sum_{|\balpha|=d_i} U_{i,\balpha} \bx^{\balpha}$$
the generic polynomial of degree $d_i\geq 1$, where $\balpha$ is a multi-index $(\alpha_0,\ldots,\alpha_n)\in \NN^{n+1}$, $\bx^\balpha$ is the monomial $x_0^{\alpha_0}\ldots x_n^{\alpha_n}$ and $|\balpha|:=\sum_{i=0}^n \alpha_i$. The universal ring of coefficients of $F_0,\ldots,F_n$ over $k$ is the ring
$$A_k:=k[U_{i,\balpha} : i=0,\ldots,n, \ |\balpha|=d_i]$$
and we set $C_k=A_k[x_0,\ldots,x_n]$. To not overload the notation, we will only display the base ring $k$ when it plays an  important role. 

The polynomial ring $C$ is canonically graded by setting $\deg(x_i)=1$.  We denote by $I$, respectively $\mm$, the homogeneous ideal of $C$ generated by $F_0,\ldots,F_n$, respectively $x_0,\ldots,x_n$. We also define the graded quotient ring $B=C/I$ and set $\delta:=\sum_{i=0}^n(d_i-1)$.

\medskip

From a geometric point of view, the ideal $I$ defines a subscheme in $\PP^n_A$ whose canonical projection on the affine (coefficient) space $\Spec(A)$ is defined by the elimination ideal $\Af=(I:\mm^\infty)\cap A$. Although this property is not needed for what follows, we notice that in our setting the ideal $\Af$ is a prime and principal ideal which is generated by the Macaulay resultant of $F_0,\ldots,F_n$ \cite{J91}. The following classical result is important as it allows to compute $\Af$ as the annihilator of some graded components of $B$ (see e.g.~\cite{J97}).

\begin{lem}\label{lem:SGsatindex} For any integer $\nu>\delta$, $\Af=\ann_A(B_\nu)$.	
\end{lem}

By classical properties of Fitting ideals, both ideals $\Fitt^0_A(B_\nu)$ and $\ann_A(B_\nu)$ of $A$ have the same radical (for all $\nu\in \NN$) and hence Lemma \ref{lem:SGsatindex} suggests to consider presentation matrices of $A$-modules $B_\nu$. Consider the graded map
\begin{eqnarray*}
\Phi: \oplus_{i=0}^n C(-d_i) & \xrightarrow{ (F_0, \ldots, F_n)} & C \\
(G_0,\ldots,G_n) & \mapsto & G_0F_0+G_1F_1+\cdots+G_nF_n.	
\end{eqnarray*}
For any integer $\nu$ we denote by $\MM_\nu$ the matrix of the graded component map $\Phi_\nu$ of $\Phi$. It is a map of $A$-modules which provides a presentation of $B_\nu$. Thus, as a consequence of Lemma \ref{lem:SGsatindex}, the ideal generated by the maximal minors of $\MM_\nu$ has the same radical as $\Af$ for all $\nu>\delta$. 

\medskip

Now, let us consider a zero-dimensional polynomial system with coefficients in a field $\KK$: $f_0,\ldots,f_n$ are $n+1$ homogeneous polynomials of degrees $d_0,\ldots,d_n$ in $R=\KK[x_0,\ldots,x_n]$. This polynomial system can be seen as a specialization of the generic polynomial system $F_0,\ldots,F_n$ via a ring map $\rho:A_\ZZ \rightarrow \KK $. In other words,  setting $\bff:=\{f_0,\ldots,f_n\}$, $\bff$ can be seen as a point in $\mathrm{Spec}(A_\KK)$. Thus, we use the notation $I(\bff)$ for the ideal generated by $f_0,\ldots,f_n$ in $R$ and similarly we will use the notation $B(\bff)$ and $\MM_\nu(\bff)$.

The graded components of $B(\bff)$ are $\KK$-vector spaces whose dimensions are defining the Hilbert function of $B(\bff)$,  $\HF_{B(\bff)}(\nu):=\dim_\KK B(\bff)_\nu$. It is known since Hilbert that for sufficiently high values of $\nu$, the Hilbert function is a polynomial function which is called the Hilbert polynomial and denoted by $\HP_{B(\bff)}(\nu)$. In the case where $I(\bff)$ defines finitely many points in $\PP^n_\KK$, the Hilbert polynomial is a constant polynomial which is equal to the number of points defined by $I(\bff)$, counted with multiplicity. 
\begin{lem}\label{lem:SGHilbReg} For any integer $\nu>\delta$ we have $\HF_{B(\bff)}(\nu)=\HP_{B(\bff)}(\nu)$.
\end{lem}
\begin{proof} This is a classical result; see for instance \cite[Theorem 3.3]{Laz} and \cite[\S 1.5]{J96}.
\end{proof}

All the above considerations lead to the following result that we will refer to as the \emph{drop-of-rank property} in the rest of this paper.

\begin{prop}\label{prop:SGKoszulMat} Suppose that the polynomial system $\bff=\{f_0,\ldots, f_n\}$ defines a finite subscheme in $\PP^n_{\bar{\KK}}$ and let $\kappa$ be its degree. Then, for all $\nu > \delta$, the dimension of the cokernel of $\MM_\nu(\bff)$ is equal to $\kappa$.
\end{prop}
\begin{proof} This is a immediate consequence of the right exactness of the tensor product and Lemma \ref{lem:SGHilbReg}.
\end{proof}

Recall that the degree of a finite subscheme is the sum of the length of the local rings of the points. In particular, it is equal to zero if the subscheme is empty.

This proposition is the key property to use the matrix $\MM_\nu(\bff)$ to solve the polynomial system $\bff=0$ by means of linear algebra techniques, in particular singular value decompositions and eigenvalue computations. We notice that in general the matrices  $\MM_\nu(\bff)$ are not square, except for rare exceptions, as the following one. 

\begin{exmp}
If $n=1$, then the matrix of $\MM_{\delta+1}(\bff)$ is nothing but the classical Sylvester matrix of the two polynomials $f_0$ and $f_1$. Thus, Proposition \ref{prop:SGKoszulMat} is the well known property that the corank of $\MM_{\delta+1}(\bff)$ is equal to the degree of the greatest common divisor of $f_0$ and $f_1$. 
\end{exmp}

\subsection{Saturation and duality} \label{section Sat and duality}
To construct smaller matrices having the drop-of-rank property, one possibility is to add new equations to our polynomial system. As we do not want to change its geometry, we consider the ideal $I^{\sat}$ which is obtained by saturation of the ideal $I$ with respect to the irrelevant ideal $\mm$ in $C$, i.e.
$$ I^{\sat}=(I:\mm^{\infty})=\{p \in C : \exists k\in \ZZ \  \mm^k p \subset I \} \subset C.$$ 
We set $B^{\sat}=C/I^{\sat}$ and we adopt the following notation for specialized polynomial systems over a field $\KK$. Let $\bff$ be a polynomial system with coefficients in $\KK$, we denote by $B(\bff)^{\sat}$ the quotient ring $R/I(\bff)^{\sat}$, where the saturation is taken after specialization, and by $B^{\sat}(\bff)$ the quotient ring $R/I^{\sat}(\bff)$, where the saturation is taken in the generic setting and then specialized. The following result is an improvement of Lemma \ref{lem:SGHilbReg}. 
 
\begin{lem}\label{lem:SGregIsat} Assume that $I(\bff)$ defines finitely many points in $\PP^n_{\bar{k}}$, say $\kappa \in \NN$,  then for any integer $\nu>\delta-\min_i d_i$, $\HF_{B(\bff)^{\sat}}(\nu)=\HF_{B^{\sat}(\bff)}(\nu)=\HP_{B(\bff)}(\nu)=\kappa$.
\end{lem}
\begin{proof} A first observation is that the three ideals $I(\bff)\subset I^{\sat}(\bff) \subset I(\bff)^{\sat}$ of $R$ have the same saturation so $B(\bff)^{\sat}$, $B^{\sat}(\bff)$ and $B(\bff)$ have the same Hilbert polynomial.
	
By the Grothendieck-Serre formula \cite[Theorem 4.3.5]{BrHe93}, since $H^i_\mm(B(\bff)^{\sat})=0$ for all $i>1$ by our assumption, we have  $\HF_{B(\bff)^{\sat}}(\nu)=\HP_{B(\bff)^{\sat}}(\nu)$ for all $\nu$ such that $H^1_\mm(B(\bff)^{\sat})_\nu=0$ (observe that $H^0_\mm(B(\bff)^{\sat})=0$). The fact that this latter condition is satisfied for all $\nu>\delta-\min_i d_i$ follows classically from the analysis of the two \v{C}ech-Koszul spectral sequences associated with $I(\bff)$; see for instance \cite[\S 2.11]{J80}. 

Similarly, we have $\HF_{B^{\sat}(\bff)}(\nu)=\HP_{B^{\sat}(\bff)}(\nu)$ for all $\nu$ such that $H^0_\mm(B^{\sat}(\bff))_\nu=0$ and $H^1_\mm(B^{\sat}(\bff))_\nu=0$. The vanishing of these two local cohomology modules can be controlled as fibers of projective morphisms. More precisely, by \cite[Proposition 6.3]{Ch13} we deduce that $H^0_\mm(B^{\sat}(\bff))_\nu=0$ and $H^1_\mm(B^{\sat}(\bff))_\nu=0$ for all $\nu$ such that $H^0_\mm(B^{\sat})_\nu=0$ and $H^1_\mm(B^{\sat})_\nu=0$. But $H^0_\mm(B^{\sat})=0$ and $H^1_\mm(B^{\sat})_\nu=0$ for all $\nu>\delta-\min_i d_i$ by the analysis of the two \v{C}ech-Koszul spectral sequences associated with $I$, as we already mentioned, which concludes the proof.
\end{proof}

\begin{rem} As a consequence of Lemma \ref{lem:SGregIsat}, we notice that the canonical map from $I^{\sat}_\nu$ to $I(\bff)^{\sat}_\nu$, which is induced by the specialization $C\rightarrow R$ sending $F_i$ to $f_i$, is surjective for all $\nu > \delta-\min_i d_i$. 	
\end{rem}

To take advantage of Lemma \ref{lem:SGregIsat}, we need to understand the graded components of ${B^{\sat}}$, equivalently ${I^{\sat}}$, for all degree $\nu>\delta-\min_i d_i$. Since $I\subset I^{\sat}$, it is sufficient to analyze the quotient $I^{\sat}/I$. The following classical duality property is a key point \cite{J96,SS01}.

\begin{prop}\label{prop: SGduality} For any integer $\nu>\delta$, $\left(I^{\sat}/I\right)_\nu=0$. In addition, for any integer $0\leq \nu \leq \delta$ there is a duality of $A_k$-modules
	$$\Hom_{A_k}(B_\nu,A_k) \xrightarrow{\sim} \left(I^{\sat}/I\right)_{\delta-\nu}.$$	
\end{prop}

The duality maps in this proposition can described in terms of Morley forms, as proved by Jouanolou in \cite[\S 3.11]{J97}. In particular, for all $0\leq \nu < \min_i d_i$ the graded components $\left(I^{\sat}/I\right)_{\delta-\nu}$ are isomorphic to $C_\nu$ and $A$-bases of these graded components are provided by Sylvester forms.   

\subsection{Sylvester forms}\label{subsec:SGSylvForm}

We suppose given a multi-index $\bgamma:=(\gamma_0,\ldots,\gamma_{n})$, $\gamma_i \in \mathbb{N}$ such that $|\bgamma|=\sum_{i=0}^n \gamma_i <\min_j d_{j}$. Under this assumption, one can decompose any polynomial $F_{i}$ under the form
\begin{equation}\label{eq:choicedecompfi}
F_{i}=x_{0}^{\gamma_0+1} F_{i,0}+x_{1}^{\gamma_1+1} F_{i,1}+\cdots +x_{n}^{\gamma_{n}+1} F_{i,n},
\end{equation}
where the polynomials $F_{i,k}$ are homogeneous of degree $d_{i}-\gamma_k-1$ in $C$. Following \cite[\S 3.10]{J97}, we define the Sylvester form of the polynomials $F_{0},\ldots,F_{n}$ in degree $\bgamma$ as the determinant
$$\Sylv_\bgamma:=
\det\left( 
\begin{array}{ccc}
F_{0,0} & \cdots & F_{n,0} \\
\vdots & & \vdots \\
F_{0,n} & \cdots & F_{n,n}
\end{array}
\right).
$$
By construction, the Sylvester form $\Sylv_\bgamma$ belongs to the ideal $I^{\sat}$ and is of degree $\delta-|\bgamma|$. It depends on the choice of decompositions \eqref{eq:choicedecompfi}, but its class in $I^{\sat}/I=H^0_\mm(B)$, which we denote by $\sylv_\bgamma$, does not depend on these choices; see \cite[\S 3.10.1]{J97}. We have the following property which is of particular interest for our purposes.
\begin{prop}\label{prop:SGdualitySF}
For any multi-indexes $\bgamma,\bgamma'$ such that $|\bgamma|=|\bgamma'|<\min_i d_{i}$ we have
\begin{equation}\label{equation-single}
 \bx^{\bgamma'}\sylv_{\bgamma}= 
\begin{cases}
	\sylv_{(0,\ldots,0)} & \mathrm{\  if \ } \bgamma=\bgamma' \\
	0 & \mathrm{\  otherwise. \ }
\end{cases}
\end{equation}
In addition, the set of Sylvester forms $\{\sylv_\bgamma\ : |\bgamma|=\delta-\nu\}$ yields an $A_k$-basis of the graded component $\left(I^{\sat}/I\right)_{\nu}$ for all $\nu>\delta-\min_i d_i$.
\end{prop}

The form $\sylv_{(0,\ldots,0)}$ plays a particular role: it is a generator of $\left(I^{\sat}/I\right)_{\delta}$,  which is a free $A$-module of rank one, and it is actually proportional to the more classical Jacobian determinant of the polynomials $F_0,\ldots,F_n$. 

We notice that Sylvester forms are, by construction, universal in the coefficients of each polynomial $F_i$ and are actually linear in each of these sets of coefficients. For any polynomial system $\bff$ with coefficients over a field $\KK$, we denote by $\Sylv_\bgamma(\bff)$ and $\sylv_\bgamma(\bff)$ the corresponding specialized Sylvester forms.

\subsection{Hybrid elimination matrices}

From the previous results one can extend the family of elimination matrices $\MM_\nu$ and get some more compact ones. For that purpose, for any integer $ \nu > \delta- \min_i d_i$ consider the $A$-modules morphism
\begin{eqnarray}\label{eq:SGpsi}
\Psi_\nu : \oplus_{i=0}^n C_{\nu-d_i}  \oplus_{\bgamma : |\bgamma|=\delta-\nu} A & \rightarrow & C_\nu \\ \nonumber
(G_0,\ldots,G_n, \ldots,\ell_\bgamma,\ldots) & \mapsto & \sum_{i=0}^n G_iF_i + \sum_{|\bgamma|=\delta-\nu}\ell_\bgamma \Sylv_{\bgamma}.
\end{eqnarray}
We denote by $\HH_\nu$ the matrix of \eqref{eq:SGpsi} in canonical bases. This matrix is made of two column blocks;  the matrix $\MM_\nu$ defined in Section \eqref{subsec:MacMat} and  the coefficient matrix of the Sylvester forms of degree $\delta-\nu$. In particular, if $\nu> \delta$ then the second block vanishes and $\HH_\nu=\MM_\nu$. Therefore, the family of matrices $\HH_\nu$ can be seen as an extension of the family of matrices $\MM_\nu$ that is valid for integers $\nu$ such that $\delta-\min_i d_i < \nu \leq \delta$. 

As a consequence of Proposition \ref{prop:SGdualitySF}, $\HH_\nu$ is a presentation matrix of $B^{\sat}_\nu=C_\nu/I^{\sat}_\nu$. Hence, the ideal generated by the maximal minors of $\HH_\nu$ has the same radical as the elimination ideal $\Af$. Moreover, the matrices $\HH_\nu$ have the expected drop-of-rank property. To be more precise, observe that, by the construction,  the matrices $\HH_\nu$ are universal in the coefficients of $F_0,\ldots,F_n$. Hence, given a polynomial system $\bff=\{f_0,\ldots,f_n\}$ with coefficients in a field $\KK$ we denote by $\HH_\nu(\bff)$ the matrix $\HH_\nu$ specialized to the system $\bff$. We did not find the following result in the existing literature, although all the necessary ingredients to prove it are known (we notice that the matrix $\HH_\delta$ already appeared in \cite[Proposition 2.1]{CDS98}).

\begin{prop}\label{prop:SGHybMat} Assume that the polynomial system $\bff$ defines a finite subscheme in $\PP^n_{\bar{\KK}}$ of degree $\kappa$. Then, for all $\nu > \delta - \min_i d_i$, the corank of $\HH_\nu(\bff)$ is equal to $\kappa$.
\end{prop}
\begin{proof} By construction, the cokernel of $\HH(\bff)_\nu$ is isomorphic to $B^{\sat}(\bff)_\nu$ and hence the  claimed result follows from Lemma \ref{lem:SGregIsat}.
\end{proof}

We notice that the matrix $\HH_{\delta-\min_i d_i+1}$ is of smaller size than the matrix $\MM_{\delta+1}=\HH_{\delta+1}$ and it can be used in a similar way to solve the polynomial system $f_0=\cdots=f_n=0$ without changing its geometric structure. More precisely, assuming $d_0\geq d_1\geq \ldots\geq d_n$, the number of rows of $\psi_{\delta+1-\min_i d_i}$ is equal to 
$\binom{\sum_{i=0}^{n-1}d_i}{n}$
whereas the number of rows of the more commonly used matrix  $\phi_{\delta+1}$ is equal to 
$\binom{ \sum_{i=0}^{n}d_i}{n}.$ We also mention that for $\nu \leq \delta$ we call the matrices $\HH_\nu$ ``hybrid elimination matrices'' because of the following example. 
 
\begin{exmp} If $n=1$ we have $\delta=d_0+d_1-2$ and the matrix $\MM_{d_1+d_2-1}$ is the Sylvester matrix whose determinant is the  Sylvester resultant  of $f_0$ and $f_1$. The matrices $\HH_{\nu}$, with $\max_i d_i \leq \nu \leq \delta$, corresponds to the well known hybrid Bézout matrices; see for instance \cite{SGD97,DG02}.	
\end{exmp}


In the rest of this paper, we will generalize the above results to the multigraded case, i.e.~in the case where the polynomials $F_i$ are multihomogeneous polynomials. For that purpose we will prove a (partial) duality property and introduce multigraded Sylvester forms.

\section{Multigraded saturation and duality}\label{sec:duality}

In the previous section we considered homogeneous polynomial equations defining hypersurfaces in the projective space $\mathbb{P}^n$. From now on we will consider  multihomogeneous polynomial equations in a product of projective spaces. Our main goal in this section is to provide generalizations of Lemma \ref{lem:SGregIsat} and Proposition \ref{prop: SGduality} to this context.  We will use the following notation for the rest of the article. 

\begin{notn}\label{Notation}
Fix positive integers $n_1,\dots, n_r$ and set $n= 
n_1+\cdots + n_r$. For all $j=1,\ldots,r$ we denote by $\bx_j$ the set of variables $x_{j,0},\ldots, x_{j,n_j}$ and for all $i=0,\ldots,n$, we consider the generic multihomogeneous polynomial  of degree $\bd_i= (d_{i,1},\dots, d_{i,r})$
$$F_i(\bx_1,\ldots,\bx_r):=\sum_{|\balpha_1|=d_{i,1}, \ldots, |\balpha_r|=d_{i,r}} U_{i,\balpha_1,\ldots,\balpha_r}\bx_1^{\balpha_1}\cdots \bx_r^{\alpha_r}.$$
We define the universal ring of coefficients over the commutative ring $k$ as 
$$A_k=k[U_{i,\balpha_1,\ldots,\balpha_r} : i=0,\ldots,n, |\balpha_1|=d_{i,1}, \ldots, |\balpha_r|=d_{i,r}]$$ 
which is multigraded by setting 
$$\deg(U_{i,\balpha_1,\ldots,\balpha_r})=(0,\ldots,0,1,0,\ldots,0)=:\textbf{e}_i$$ 
(the ``1'' is at the $i^\mathrm{th}$ place). Thus, the polynomials $F_i$ are multihomogeneous polynomials in the polynomial ring  $C:=A_k[\bx_1,\ldots,\bx_r]$: 
$$\deg_{\bx_j} F_i=d_{i,j}\geq 1 \textrm{ and } \deg_{\bU} F_i=\mathbf{e}_i,$$
where $\textbf{U}$ denotes the set of all coefficients of all the polynomials $F_0,\ldots,F_n$.
We define the ideals $\bb= \prod_{i=1}^r \mm_i$ where $\mm_i:=(\bx_i)$ for all $i$, $\mm:= \sum_{i=1}^r \mm_i \subset C$ and $I:= (F_0,\dots , 
F_{n})\subset C$, as well as the quotient ring $B:= C/I$. For all $j=1,\ldots,r$ we define $\delta_j:=(\sum_{i=0}^n d_{i,j})-(n_j+1)$ and we set
\begin{equation}\label{eq:delta}
\bdelta = (\delta_1,\dots , \delta_r) = \sum_{i=0}^{n}\bd_i-(n_1+1,n_2+1,\dots, n_r+1).	
\end{equation}
\end{notn}
 Geometrically, the multigraded ring $C$ can be interpreted as the coordinate ring of the product of projective schemes $\PP_{A_k}=\PP^{n_1}_{A_k}\times \ldots \times \PP^{n_r}_{A_k}$.  Thus,  $\bb$ is the irrelevant ideal of $\PP_{A_k}$. The ideal $I$ defines a subscheme in $\PP_{A_k}$ whose canonical projection on the affine  space $\Spec(A_k)$ is defined by the elimination ideal $\Af=(I:\bb^\infty)\cap A_k$.  Notice that, as expected, this elimination ideal is equal to $0$ if the number of generators of $I$ is less than or equal $n$ (see Corollary \ref{cor: elimination ideal}). Therefore, having $n+1$ polynomials as the generators of $I$ is the first interesting case of study for $\Af$.

\medskip

We begin with two technical lemmas that are taken from \cite[Chapter I, Proposition 3.1.2]{ChThesis88}. We reproduce their proofs for the sake of completeness and accessibility. These lemmas are extensions of well-known results in the single graded case, see for instance \cite[\S 4.2 and \S 4.7]{J91}, to the multigraded case.

\begin{lem}\label{lem: appendix}
For any sequence of $r$ integers $(i_1,\dots,i_r)$ such that $0\leq i_j\leq n_j$ for all $j=1,\ldots,r$, there exists an isomorphism of $A'_k[\bx_1,\dots,\bx_r]$-algebras
$$
B_{x_{1,i_1}\ldots x_{r,i_r}}\xrightarrow{\sim} A'_k[\bx_1,\dots,\bx_r]_{x_{1,i_1}\cdots x_{r,i_r}}
$$
where $A'_{k}= k[U_{i, \balpha_1,\dots, \balpha_r} \vert \, U_{i, \balpha_1,\dots, \balpha_r} \neq U_{i, i_1,\dots, i_r} \,\, 0\leq i\leq n]$.
\end{lem}
\begin{proof}
For simplicity, we consider the case $(i_1,\ldots,i_r)=(n_1,\ldots,n_r)$ and we set $\sigma:=x_{1,n_1}\ldots x_{r,n_r}$. For all $i=0,\ldots,n$ we also set  
$\epsilon_i:=U_{i, n_1,\dots, n_r}$ and we define the monomial $\tau_i= x_{1,n_1}^{d_{i,1}}\cdots x_{r,n_r}^{d_{i,r}}$ and the polynomial $H_i := F_i- \epsilon_i\tau_i$. Notice that $A_k=A'_{k}[\epsilon_1,\dots, \epsilon_r]$. Now, consider the morphism $\phi$ of $A'_{k}[\bx_1,\dots \bx_r]$- algebras
\begin{align*}
\phi: A_k[\bx_1,\dots,\bx_r]&\rightarrow A'_{k}[\bx_1,\dots \bx_r]_{\sigma}\\
\epsilon_i&\mapsto - {H_i}/{\tau_i}
\end{align*}
which leaves invariant all variables and all coefficients except the $\epsilon_i$'s. Since  
$\phi(F_i)=0$ for all $i$, $\phi$ induces the claimed isomorphism of $A'_{k}[\bx_1,\dots \bx_r]$- algebras.
\end{proof}

\begin{rem}\label{rem:TFmon}
	A consequence of the above lemma is that for any choice of sequences of integers $(i_1,\ldots,i_r)$ and $(i_1',\ldots,i_r')$, the monomial $x_{1,i_1'}\cdots x_{r,i_r'}$ is not a zero divisor in $B_{x_{1,i_1}\ldots x_{r,i_r}}$. In particular, if the class of $P\in A_k[\bx_1,\dots \bx_r]$ is equal to 0 in $B_{x_{1,i_1}\ldots x_{r,i_r}}$ then it is also equal to 0 in $B_{x_{1,i_1'}\ldots x_{r,i_r'}}$.
\end{rem}

\begin{lem}\label{lem-reqular-sequence}
The generic multihomogeneous polynomials $F_0$, $\ldots$, $F_{n}$ form a regular sequence in   $C_{x_{1,i_1}\ldots x_{r,i_r}}$ for any sequence $(i_1,\dots,i_r)$ of $r$ integers such that $0\leq i_j\leq n_j$ for all $j=1,\ldots,r$.
\end{lem}
\begin{proof} As in the proof of Lemma \ref{lem: appendix}, we treat the case $(i_1,\ldots,i_r)=(n_1,\ldots,n_r)$ for simplicity and we set $\sigma:=x_{1,n_1}\ldots x_{r,n_r}$.
	
To begin with, we claim that $F_0$ is a nonzero divisor in $C$. To see it we use the following corollary of Dedekind-Mertens Lemma (see \cite[Corollary 2.8]{Buse-Jouanolou}): a polynomial $F$ is a nonzero divisor in $C$ if and only if its content ideal (i.e.~the ideal in $A_k$ generated by the coefficients of $F$) does not divide zero in $A_k$. Thus, as any coefficient $U_{i, \balpha_1,\dots, \balpha_r}$ is a  nonzero divisor in $A_k$,  we deduce that $F_0$ is a nonzero divisor in $C$, hence in $C_{\sigma}$. 

Now, set 
$\epsilon_i:=U_{i, n_1,\dots, n_r}$ for all $i=0,\ldots,n$, let $t$ be an integer such that $0< t< n$ and define  
$$
A^{(t)}_k := k[U_{i, \balpha_1,\dots, \balpha_r} \, \vert \,  U_{i, \balpha_1,\dots, \balpha_r}\neq \epsilon_i, \, 0\leq i\leq t],
$$
so that $A_k = A^{(t)}_k[\epsilon_1,\ldots ,\epsilon_t]$. According to Lemma \ref{lem: appendix} (applied with $t+1$ polynomials instead of $n+1$),
$$
\left( A_k[\bx_1,\dots \bx_r]/(F_0,\dots F_t)\right)_{\sigma} \simeq A^{(t)}_k[\bx_1,\dots \bx_r]_{\sigma}
$$
and since $F_{t+1}$ is a nonzero divisor in $A^{(t)}_k[\bx_1,\dots,\bx_r]$ by the above corollary of Dedekind-Mertens Lemma, we deduce that $F_{t+1}$ is a nonzero divisor in $\left( A_k[\bx_1,\dots \bx_r]/(F_0,\dots F_t)\right)_{\sigma}\simeq A^{(t)}_k[\bx_1,\dots \bx_r]_{\sigma}$.
\end{proof}

After these preliminaries, our next task is to provide the precise statement and proof of Theorem A. As our strategy relies on the analysis of some local cohomology modules, we first introduce additional notation in order to describe the local cohomologies of the polynomial ring $C$ with respect to $\bb$.

\begin{defn} For all $j=1,\ldots,r$, set $C_j=A_k[\bx_j]$ and define the $C_j$-module
	$$\check{C}_j:= H^{n_j+1}_{\mm_j}(C_j) = \dfrac{1}{x_{j,0}\cdots x_{j,n_j}}A_k[x_{j,0}^{-1},\dots, x_{j,n_j}^{-1}],$$
which is canonically $\ZZ$-graded. For any subset $\alpha := \{i_1,\dots, i_t\}\subseteq \{1,\dots,r\}$ such that $1\leq i_1<\cdots< i_t\leq r$ we define the $C$-module
$$\check{C}_{\alpha}:= 
D_{1}\otimes_{A_k} \cdots \otimes_{A_k} D_{r} \ \textrm{ with } D_j:=C_j \textrm{ if } j\notin \alpha \textrm{ and } D_j:=\check{C}_j \textrm{ else},
$$
which is canonically $\ZZ^r$-graded (recall $C=C_1\otimes_{A_k} \cdots \otimes_{A_k} C_r$). Finally, for every $\alpha\neq \emptyset$ we define 
$$Q_{\alpha}:= \Supp(\check{C}_{\alpha})=\{\bmu \in \ZZ^r \, : \, (\check{C}_{\alpha})_\bmu\neq 0 \} \subset \ZZ^r,$$ 
with the convention $Q_{\emptyset}:= \emptyset$.
\end{defn}

\begin{prop}\label{prop: support cohomology}
With the above notation, the following properties hold:
\begin{itemize}
\item[$(1)$]$ H^{\ell}_{\bb}(C) \cong \bigoplus_{\substack{\alpha \subseteq \{ 1,\dots,r\}\\ n(\alpha)+1=\ell}} \check{C}_{\alpha}
$, where $n(\alpha) = \sum_{j\in \alpha} n_j$.
\item[$(2)$] Let $\alpha = \{ i_1,\dots,i_t\}$ such that $1\leq i_1<\cdots < i_t\leq r$, then 
$$
Q_{\alpha}= \bigoplus_{1\leq j\leq r} ((-1)^{\mathrm{sgn}
_{\alpha}(j)}\N- \mathrm{sgn}_{\alpha}(j)(n_j+1))\cdot \textbf{e}_j 
\subset \mathbb{Z}^r,
$$
where $\mathrm{sgn}_{\alpha}(j)= 1$ if $j\in \alpha$ and $\mathrm{sgn}_{\alpha}(j)=0$ if $j\notin \alpha$. 
\end{itemize}
\end{prop}
\begin{proof}
See Lemma 6.5 and Lemma 6.7 in \cite{botbol}.
\end{proof}

\begin{exmp}\label{ex: top cohomology}
With the above notation,  $H^{n+1}_{\bb}(C) = \check{C}_{\alpha}$ where $\alpha=\{1,\dots,r\}$. In addition, 
$$
\Supp(H^{n+1}_{\bb}(C)) = Q_{\{1,\dots,r\}} = -(n_1+1,\dots, n_r+1)-\NN^r= \bdelta+ \sum_{i=0}^n\bd_i- \NN^r.$$
\end{exmp}

We will analyze the support of the local cohomology modules of the terms of the multigraded Koszul complex associated with the sequence of multihomogeneous polynomials $F_0,\ldots,F_n$ in $C$. We denote this complex by $K_\bullet(\mathbf{F},C)$ and for all $i=0,1,2$ we set
$$\Gamma_i := \bigcup_{-1\leq p\leq n-1}  \Supp \left(H^{p+1}_{\bb}\left(K_{p+i}\left(\bF,C\right)\right)\right).$$
Recall that for subsets $A,B\subset \ZZ^r$ and $c\in \ZZ$,  
$$
A+ B := \{ a + b \, | \, a\in A \,\, \text{and}\,\, b\in B\} \,\, \text{and}\,\, c\cdot A := \{ c\cdot a \, | \, a\in A \}.
$$
Additionally, given a finite set of integers $\alpha \subset \NN$ we denote by $\sharp \alpha$ its number of elements. 

\begin{cor}\label{cor:Gammai} For $i=0,1,2$, the following equality of subsets in $\ZZ^r$ holds:
	$$
	\Gamma_i 
	=\bigcup_{\substack{\alpha \subset \{1,\ldots,r\}, \,  1\leq \sharp\alpha\leq r-1\\ 
	\lambda \subset \{0,\ldots,n\}, \, \sharp\lambda = n(\alpha)+i}} \ \left(\sum_{j\in \lambda} \bd_j + Q_{\alpha} \right), 
	$$
	where $n(\alpha)= \sum_{j\in \alpha} n_j$.	
\end{cor}
\begin{proof} We prove this formula in the case $r=2$ as the case $r>2$ goes along the same lines. From Proposition \ref{prop: support cohomology} we deduce
	$$
	Q_{\{1\}}= \Supp(\check{C}_{\{1\}}) = \Supp (H^{n_1+1}_{\bb}(C)), \,\,
	Q_{\{2\}}= \Supp(\check{C}_{\{2\}})= \Supp (H^{n_2+1}_{\bb}(C)).
	$$
	Then, by Mayer-Vietoris exact sequence we obtain the isomorphisms $H^{n_1+1}_{\bb}(C)\cong H^{n_1+1}_{\mm_1}(C)$ and $H^{n_2+1}_{\bb}(C)\cong H^{n_2+1}_{\mm_2}(C)$, unless $n_1=n_2$ in which case $$H^{n_1+1}_{\bb}(C)\cong H^{n_1+1}_{\mm_1}(C)\oplus H^{n_2+1}_{\mm_2}(C).$$ Therefore,  
	\begin{align*}
	\Gamma_i 
	& = \Supp(H^{n_1+1}_{\mm_1}(K_{n_1+i}(\bff,C)))\cup \Supp(H^{n_2+1}_{\mm_2}(K_{n_2+i}(\bff,C)))\\
	&= \left(\sum_{\substack {j\in \lambda, \\ \sharp\lambda = n_1+1}} \bd_j + \Supp (H^{n_1+1}_{\mm_1}(C)) \right) \bigcup  \left( \sum_{\substack {j\in \lambda, \\ \sharp\lambda = n_2+1}} \bd_j + \Supp (H^{n_2+1}_{\mm_2}(C))\right)\\
	& = \left(\sum_{\substack {j\in \lambda, \\ \sharp \lambda = n_1+1}} \bd_j + Q_{1} \right) \bigcup \left(\sum_{\substack {j\in \lambda, \\ \sharp \lambda = n_2+1}} \bd_j + Q_{2}\right).
	\end{align*}
\end{proof}

\begin{exmp} We illustrate graphically Corollary \ref{cor:Gammai} in the case $r=2$ and  $\bd_0=(2,1), \bd_1=(1,1)$ and $\bd_2=(1,1)$ with the two sets of 
variables $\bx_1= (x_{1,0},x_{1,1})$  and $\bx_2= 
(x_{2,0},x_{2,1})$ (geometrically we are over $\PP^{1}\times \PP^{1}$). 
In the following picture the red (resp. yellow, resp. green) region represents the set \red{$\Gamma_0$} (resp. \yellow{$\Gamma_1$}, resp. \green{$\Gamma_2$}).  The grey region is $\bdelta-\NN^r$ that will appear in Corollary \ref{cor- main theorem}.
\medskip
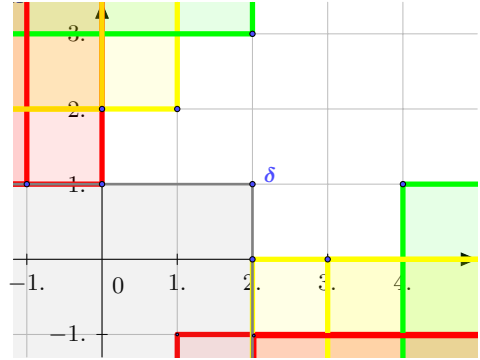
\begin{center}
\definecolor{ffdxqq}{rgb}{1.,0.8431372549019608,0.}
\definecolor{qqffqq}{rgb}{0.,1.,0.}
\definecolor{ffqqqq}{rgb}{1.,0.,0.}
\definecolor{uuuuuu}{rgb}{0.26666666666666666,0.26666666666666666,0.26666666666666666}
\definecolor{ffffqq}{rgb}{1.,1.,0.}
\definecolor{ududff}{rgb}{0.30196078431372547,0.30196078431372547,1.}
\definecolor{cqcqcq}{rgb}{0.7529411764705882,0.7529411764705882,0.7529411764705882}
\begin{tikzpicture}[line cap=round,line join=round,>=triangle 45,x=1.0cm,y=1.0cm]
\draw [color=cqcqcq,, xstep=1.0cm,ystep=1.0cm] (-1.177832074545664,-1.3001504352679525) grid (4.988286119555655,3.4201193546992252);
\draw[->,color=black] (-1.177832074545664,0.) -- (4.988286119555655,0.);
\foreach \x in {-1.,1.,2.,3.,4.}
\draw[shift={(\x,0)},color=black] (0pt,2pt) -- (0pt,-2pt) node[below] {\footnotesize $\x$};
\draw[->,color=black] (0.,-1.3001504352679525) -- (0.,3.4201193546992252);
\foreach \y in {-1.,1.,2.,3.}
\draw[shift={(0,\y)},color=black] (2pt,0pt) -- (-2pt,0pt) node[left] {\footnotesize $\y$};
\draw[color=black] (0pt,-10pt) node[right] {\footnotesize $0$};
\clip(-1.177832074545664,-1.3001504352679525) rectangle (4.988286119555655,3.4201193546992252);
\fill[line width=3.2pt,color=ffffqq,fill=ffffqq,fill opacity=0.10000000149011612] (1.,2.) -- (1.0025348000000003,8.376970400000008) -- (-5.374435600000007,8.379505200000008) -- (-5.376970400000008,2.0025348000000007) -- cycle;
\fill[line width=0.4pt,color=ffqqqq,fill=ffqqqq,fill opacity=0.1] (0.,1.) -- (0.,9.) -- (-8.,9.) -- (-8.,1.) -- cycle;
\fill[line width=2.pt,color=ffqqqq,fill=ffqqqq,fill opacity=0.1] (1.,-1.) -- (1.,-11.) -- (11.,-11.) -- (11.,-1.) -- cycle;
\fill[line width=2.pt,color=ffffqq,fill=ffffqq,fill opacity=0.1] (3.,0.) -- (3.,-8.) -- (11.,-8.) -- (11.,0.) -- cycle;
\fill[line width=2.pt,color=qqffqq,fill=qqffqq,fill opacity=0.1] (2.,3.) -- (2.,12.7382) -- (-7.7382,12.7382) -- (-7.7382,3.) -- cycle;
\fill[line width=2.pt,color=qqffqq,fill=qqffqq,fill opacity=0.1] (4.,1.) -- (4.,-12.) -- (17.,-12.) -- (17.,1.) -- cycle;
\fill[line width=2.pt,color=ffffqq,fill=ffffqq,fill opacity=0.1] (2.,0.) -- (2.,-8.) -- (10.,-8.) -- (10.,0.) -- cycle;
\fill[line width=2pt,color=ffqqqq,fill=ffqqqq,fill opacity=0.1] (-1.,1.) -- (-1.,9.) -- (-9.,9.) -- (-9.,1.) -- cycle;
\fill[line width=2.pt,color=ffdxqq,fill=ffdxqq,fill opacity=0.1] (0.,2.) -- (-0.008254378512182858,8.390202772952907) -- (-6.398457151465088,8.381948394440725) -- (-6.390202772952907,1.9917456214878184) -- cycle;
\fill[line width=2.pt,color=ffqqqq,fill=ffqqqq,fill opacity=0.1] (2.01770297517665,-1.015242770886749) -- (2.01770297517665,-11.015242770886754) -- (12.017702975176654,-11.015242770886754) -- (12.017702975176656,-1.0152427708867506) -- cycle;
\draw [line width=2.pt,color=ffffqq] (1.,2.)-- (1.0025348000000003,8.376970400000008);
\draw [line width=2pt,color=ffffqq] (1.0025348000000003,8.376970400000008)-- (-5.374435600000007,8.379505200000008);
\draw [line width=2pt,color=ffffqq] (-5.374435600000007,8.379505200000008)-- (-5.376970400000008,2.0025348000000007);
\draw [line width=2pt,color=ffffqq] (-5.376970400000008,2.0025348000000007)-- (1.,2.);
\draw [line width=2pt,color=ffqqqq] (0.,1.)-- (0.,9.);
\draw [line width=2pt,color=ffqqqq] (0.,9.)-- (-8.,9.);
\draw [line width=2pt,color=ffqqqq] (-8.,9.)-- (-8.,1.);
\draw [line width=2pt,color=ffqqqq] (-8.,1.)-- (0.,1.);
\draw [line width=2.pt,color=ffqqqq] (1.,-1.)-- (1.,-11.);
\draw [line width=2.pt,color=ffqqqq] (1.,-11.)-- (11.,-11.);
\draw [line width=2.pt,color=ffqqqq] (11.,-11.)-- (11.,-1.);
\draw [line width=2.pt,color=ffqqqq] (11.,-1.)-- (1.,-1.);
\draw [line width=2.pt,color=ffffqq] (3.,0.)-- (3.,-8.);
\draw [line width=2.pt,color=ffffqq] (3.,-8.)-- (11.,-8.);
\draw [line width=2.pt,color=ffffqq] (11.,-8.)-- (11.,0.);
\draw [line width=2.pt,color=ffffqq] (11.,0.)-- (3.,0.);
\draw [line width=2.pt,color=qqffqq] (2.,3.)-- (2.,12.7382);
\draw [line width=2.pt,color=qqffqq] (2.,12.7382)-- (-7.7382,12.7382);
\draw [line width=2.pt,color=qqffqq] (-7.7382,12.7382)-- (-7.7382,3.);
\draw [line width=2.pt,color=qqffqq] (-7.7382,3.)-- (2.,3.);
\draw [line width=2.pt,color=qqffqq] (4.,1.)-- (4.,-12.);
\draw [line width=2.pt,color=qqffqq] (4.,-12.)-- (17.,-12.);
\draw [line width=2.pt,color=qqffqq] (17.,-12.)-- (17.,1.);
\draw [line width=2.pt,color=qqffqq] (17.,1.)-- (4.,1.);
\draw [line width=2.pt,color=ffffqq] (2.,0.)-- (2.,-8.);
\draw [line width=2.pt,color=ffffqq] (2.,-8.)-- (10.,-8.);
\draw [line width=2.pt,color=ffffqq] (10.,-8.)-- (10.,0.);
\draw [line width=2.pt,color=ffffqq] (10.,0.)-- (2.,0.);
\draw [line width=2pt,color=ffqqqq] (-1.,1.)-- (-1.,9.);
\draw [line width=2pt,color=ffqqqq] (-1.,9.)-- (-9.,9.);
\draw [line width=2pt,color=ffqqqq] (-9.,9.)-- (-9.,1.);
\draw [line width=2pt,color=ffqqqq] (-9.,1.)-- (-1.,1.);
\draw [line width=2pt,color=ffdxqq] (0.,2.)-- (-0.008254378512182858,8.390202772952907);
\draw [line width=2pt,color=ffdxqq] (-0.008254378512182858,8.390202772952907)-- (-6.398457151465088,8.381948394440725);
\draw [line width=2pt,color=ffdxqq] (-6.398457151465088,8.381948394440725)-- (-6.390202772952907,1.9917456214878184);
\draw [line width=2pt,color=ffdxqq] (-6.390202772952907,1.9917456214878184)-- (0.,2.);
\draw [line width=2.pt,color=ffqqqq] (2.01770297517665,-1.015242770886749)-- (2.01770297517665,-11.015242770886754);
\draw [line width=2.pt,color=ffqqqq] (2.01770297517665,-11.015242770886754)-- (12.017702975176654,-11.015242770886754);
\draw [line width=2.pt,color=ffqqqq] (12.017702975176654,-11.015242770886754)-- (12.017702975176656,-1.0152427708867506);
\draw [line width=2.pt,color=ffqqqq] (12.017702975176656,-1.0152427708867506)-- (2.01770297517665,-1.015242770886749);
\draw [line width=1.pt,color=gray] (2,1.)-- (2.,-2.);
\draw [line width=1.pt,color=gray] (2,1.)-- (-2.,1.);
\fill[line width=2.pt,color=gray,fill=gray,fill opacity=0.10000000149011612] (2.,-2.) -- (2.,1.) -- (-2.,1.) -- (-2.,-2.) -- cycle;
\begin{scriptsize}
\draw [fill=ududff] (2.,1.) circle (1.0pt);
\draw[color=ududff] (2.2347954087414448,1.1290875084426333) node {$\bdelta$};
\draw [fill=ududff] (0.,1.) circle (1.0pt);
\draw [fill=ududff] (1.,-1.) circle (0.5pt);
\draw [fill=ududff] (1.,2.) circle (1.0pt);
\draw [fill=ududff] (3.,0.) circle (1.0pt);
\draw [fill=ududff] (2.,3.) circle (1.0pt);
\draw [fill=ududff] (4.,1.) circle (1.0pt);
\draw [fill=ududff] (1.0025348000000003,8.376970400000008) circle (2.5pt);
\draw [fill=uuuuuu] (-5.374435600000007,8.379505200000008) circle (2.5pt);
\draw [fill=uuuuuu] (-5.376970400000008,2.0025348000000007) circle (2.5pt);
\draw [fill=ududff] (0.,9.) circle (2.5pt);
\draw[color=ududff] (-1.1459383597485884,3.510484879957606) node {$K$};
\draw [fill=uuuuuu] (-8.,9.) circle (2.5pt);
\draw [fill=uuuuuu] (-8.,1.) circle (2.5pt);
\draw[color=uuuuuu] (-1.1459383597485884,3.510484879957606) node {$M$};
\draw [fill=ududff] (1.,-11.) circle (0.5pt);
\draw [fill=uuuuuu] (11.,-11.) circle (0.5pt);
\draw [fill=uuuuuu] (11.,-1.) circle (0.5pt);
\draw [fill=ududff] (3.,-8.) circle (0.5pt);
\draw [fill=uuuuuu] (11.,-8.) circle (2.5pt);
\draw [fill=uuuuuu] (11.,0.) circle (0.5pt);
\draw [fill=ududff] (2.,12.7382) circle (2.5pt);
\draw [fill=uuuuuu] (-7.7382,12.7382) circle (2.5pt);
\draw [fill=uuuuuu] (-7.7382,3.) circle (2.5pt);
\draw [fill=ududff] (4.,-12.) circle (2.5pt);
\draw [fill=uuuuuu] (17.,-12.) circle (2.5pt);
\draw [fill=uuuuuu] (17.,1.) circle (2.5pt);
\draw [fill=ududff] (2.,0.) circle (1.0pt);
\draw [fill=ududff] (2.,-8.) circle (0.5pt);
\draw [fill=uuuuuu] (10.,-8.) circle (2.5pt);
\draw [fill=uuuuuu] (10.,0.) circle (0.5pt);
\draw [fill=ududff] (-1.,1.) circle (1.0pt);
\draw [fill=ududff] (-1.,9.) circle (2.5pt);
\draw[color=ududff] (-1.1193602640843585,3.5370629756218355) node {$K_1$};
\draw [fill=uuuuuu] (-9.,9.) circle (2.5pt);
\draw [fill=uuuuuu] (-9.,1.) circle (2.5pt);
\draw[color=uuuuuu] (-1.1193602640843585,3.5370629756218355) node {$M_1$};
\draw [fill=ududff] (0.,2.) circle (1.0pt);
\draw [fill=ududff] (-0.008254378512182858,8.390202772952907) circle (2.5pt);
\draw [fill=uuuuuu] (-6.398457151465088,8.381948394440725) circle (2.5pt);
\draw [fill=uuuuuu] (-6.390202772952907,1.9917456214878184) circle (2.5pt);
\draw [fill=ududff] (2.01770297517665,-1.015242770886749) circle (0.5pt);
\draw [fill=ududff] (2.01770297517665,-11.015242770886754) circle (0.5pt);
\draw [fill=uuuuuu] (12.017702975176654,-11.015242770886754) circle (0.5pt);
\draw [fill=uuuuuu] (12.017702975176656,-1.0152427708867506) circle (0.5pt);
\end{scriptsize}
\end{tikzpicture}
\end{center}
\end{exmp}

We are now ready to prove the main result of this section.

\begin{thm}\label{thm: duality}
Let $F_0,\dots ,F_n$  be the $n+1$ generic multihomogeneous polynomials 
of degree $\bd_0, \ldots, \bd_n$, respectively. If $\bmu \notin \Gamma_0\cup \Gamma_1 \subset \ZZ^r$, then 
$$\left( I^{\sat}/I\right)_{\bmu} \simeq \Hom_{A_k}((C/I)_{\bdelta-\bmu},A_k).$$
\end{thm}
\begin{proof}
We proceed by analyzing  the spectral sequences associated with the \v{C}ech-Koszul double complex  ${C}^{\bullet}_{\bb}({K}_{\bullet}(\bF,C))$.  If we 
start  taking homologies vertically, in the second page we get
\[ ^vE^{2}_{p,q} \left\{ \begin{array}{ll}
         H_q(H^p_{\bb}({K}_\bullet(\bF,C)))& \mbox{if $0\leq p\leq n+1 \textrm{ and } 0\leq q\leq n+1,  $}\\
       0 & \mbox {otherwise} 
       .\end{array} \right. \] 
If $p= n+1$ and $q=n+1$, then by Mayer-Vietoris exact sequence, 
\begin{align*}
 ^vE^{2}_{n+1, n+1}&\cong H_{n+1}(H^{n+1}_{\bb}(K_\bullet(\bF,C)))\\
&\cong H_{n+1}\left( H^{n+r}_{\mm}\left(K_\bullet(\bF,C)\right)\right)\\
&\cong H_{n+1}
( K_\bullet(\bF,C)\left(\bdelta\right)^{\star})\\
&\cong B^{\star}(\bdelta).
\end{align*}
The notation $(-)^\star$ stands, classically, for the graded dual (see \cite[\S 1.5, p.~33]{BrHe93} for more details). For instance, $B^\star=\oplus_{\bnu}\Hom_A(B_{-\bnu},A)$.

By definition, if $\bmu\notin \Gamma_0$ then for $0\leq p\leq n$, $\left( H^{p+1}_{\bb}(K_{p}(\bF,C))\right)_{\bmu}=0$, which means $ \left( ^vE^{2}_{p+1,p}\right)_{\bmu}=0$. The maps from  $^vE^{2}_{n+1, n+1}$ in the next pages are to $^vE^{2}_{p+1, p}$ for $0\leq p\leq n$ and no nonzero map points to $^vE^{\ell}_{n+1, n+1}$ for $\ell\geq 2$. It follows that 
$$
\left( ^vE^{2}_{n+1, n+1}\right)_{\bmu}\cong \left(^vE^{\infty}_{n+1, n+1}\right)_{\bmu}.
$$
If $\bmu\notin \Gamma_1$ then $\left( H^{p}_{\bb}(K_{p}(\bF,C))\right)_{\bmu}=0$ for $0\leq p\leq n$. Hence, $ \left( ^vE^{\infty}_{p,p}\right)_{\bmu}=0$ for $p\neq n+1$ and $( ^vE^{\infty}_{n+1, n+1})_{\bmu}\cong 
B^{\star}_{\bdelta-\bmu}=\Hom_A(B_{\bdelta-\bmu},A)$. 

If we start taking homology horizontally, the second page of the spectral sequence is:
$$\begin{matrix}
\ast&\cdots&\ast &H_1(K_\bullet(\bF,C))&I^{\sat}/I\\
0&\cdots&0&0&\ast\\
\vdots&\cdots&\vdots&\vdots&\vdots\\
0&\cdots&0&0&\ast.\\
\end{matrix}$$
Notice that vanishing of $H^1_{\bb}(H_1(K_\bullet(\bF,C)))$ follows from Lemma \ref{lem-reqular-sequence}. Finally, the claimed assertion follows from comparing the two spectral sequences. 
\end{proof}

We now derive some consequences of the above duality result.

\begin{cor}\label{cor- main theorem} $\left(I^{\sat}/I\right)_{\bdelta}\simeq A_k$ and $\left(I^{\sat}/I\right)_\bnu=0$ for all $\bnu\notin   \left(\bdelta-\NN^r \right)\cup \Gamma_0\cup \Gamma_1$.
\end{cor}
\begin{proof} This is a direct consequence of Theorem \ref{thm: duality} as $\bdelta\notin \Gamma_i$ for any $i$.
\end{proof}

 \begin{cor}\label{cor: syz}
 Let $F_0,\dots, F_n$ be the $n+1$ generic multihomogeneous polynomials. If $\bmu \notin \Gamma_2$ then $H_1({K}_\bullet(\bF,C))_{\bmu}=0$. In other words, any syzygy of $F_0,\ldots,F_{n}$ of degree $\bmu$ is a Koszul syzygy. 
 \end{cor}
 \begin{proof}
 We follow the same lines of the proof of Theorem \ref{thm: duality}.  Considering the two spectral sequences associated with the \v{C}ech-Koszul double complex  ${C}^{\bullet}_{\bb}({K}_{\bullet}(\bF,C))$. As  the length of the Koszul complex is equal to $n+1$, the vanishing of $H_{n+2}(H^{n+1}_{\bb}({K}_\bullet(\bF,C)))$ implies that 
 $$
\Supp\left( H_1({K}_\bullet(\bF,C))\right) \subseteq \bigcup_{-1\leq p\leq n-1 }  \Supp \left(H^{p+1}_{\bb}\left(K_{p+2}\left(\bF,C\right)\right)\right) = \Gamma_2.
 $$
 \end{proof}

 \begin{cor}\label{cor: elimination ideal}
Let $0\leq m\leq n-1$, $G_0,\dots, G_m$ be the $m+1$ generic multihomogeneous polynomials and $I$ be the ideal generated by the $G_i$'s. If $\bmu\notin \Gamma_1$ then $I^{\sat}_{\bmu} = I_{\bmu}$.  In particular, the elimination ideal $\Af=(I:\bb^\infty)\cap A_k$ is equal to $0$.
 \end{cor}
 \begin{proof}
Let $\textbf{G}= \{ G_0,\dots, G_m \}$ and consider the two spectral sequences associated with the \v{C}ech-Koszul double complex  ${C}^{\bullet}_{\bb}({K}_{\bullet}(\textbf{G},C))$. Since $^v E^2_{n+1,n+1}=0$ we deduce that
$$
\Supp \left(I^{\sat}/I\right) \subseteq  \bigcup_{-1\leq p\leq n-1 }  \Supp \left(H^{p+1}_{\bb}\left(K_{p+1}\left(\textbf{G},C\right)\right)\right) = \Gamma_1.$$
From here, the conclusion follows as $0 \notin \Gamma_\ell$, $\ell=0,1,2$, since $d_{i,j}\geq 1$ for all $i,j$ and one coordinate of a nonzero element in any $\Gamma_\ell$ is at least equal to one of the $d_{i,j}$'s.
 \end{proof}

In the case $r=2$ the combinatorial complexity in the control of the vanishing of the local cohomology modules of the Koszul homology stays reasonable and a more precise result than Theorem \ref{thm: duality} can be stated.

\begin{prop}\label{prop: Duality}
 With the same assumption as in Theorem \ref{thm: duality}, assume that $r=2$. If $\bmu \in \Gamma_1 \setminus \left( \Gamma_0 \cup \Gamma_2\right)$ then 
	$$ \left( I^{\sat}/I\right)_{\bmu} = (C/I)^{\star}_{\bdelta-\bmu} \oplus \left( H^{n_1+1}_{\mm_1}(K_{n_1+1}(\bF,C))\right)_{\bmu}\oplus \left(H^{n_2+1}_{\mm_2}(K_{n_2+1}(\bF,C))\right)_{\bmu}. $$
\end{prop}  
 \begin{proof}
We follow again the same proof as the one of Theorem \ref{thm: duality}. If we start taking homology vertically, the second page is
{\small
	$$\begin{matrix}
	0&0&0&0\\
	\vdots&\vdots&\vdots&\vdots\\
	V_{n_1+n_2+1}\subset H^{n_1+1}_{\bb_1}(K_{n_1+n_2+1})&\cdots&V_{1}\subset H^{n_1+1}_{\bb_1}(K_1) &V_{0}\subset H^{n_1+1}_{\bb_1}(K_0)\\\
	\vdots&\vdots&\vdots&\vdots\\
	W_{n_1+n_2+1}\subset H^{n_2+1}_{\bb_2}(K_{n_1+n_2+1})&\cdots &W_{1}\subset H^{n_2+1}_{\bb_2}(K_1) &W_{0}\subset H^{n_2+1}_{\bb_2}(K_0)\\\
	 \vdots&\vdots&\vdots\\
	M&\ast&\cdots&0
	\end{matrix}$$}

\noindent where $M\cong B^{\star}(\bdelta)$. Since $\bmu \notin \Gamma_0$, with the same argument as in the proof of Theorem \ref{thm: duality}, we deduce that 
	$$
	\left( ^vE^{\infty}_{n_1+n_2+1, n_1+n_2+1}\right)_{\bmu}= \left( ^vE^{2}_{n_1+n_2+1, n_1+n_2+1}\right)_{\bmu}= M_{\bmu}.$$

As 	$\bmu \notin (\Gamma_0\cup \Gamma_2)$, by the definition, 	for $p=n_1+1, n_2+1$ we get
$$
\left( ^vE^{2}_{p, p-1}\right)_{\bmu}= \left( ^vE^{2}_{p, p+1}\right)_{\bmu}=0.
$$
First, it implies that $\left( ^vE^{2}_{p,p}\right)_{\bmu}\cong \left(H^{p}_{\mm_1}(K_{p}(\bF,C))\right)_{\bmu}$ 	  for $p=n_1+1,n_2+1$. Second, it guarantees that there will be no non-zero map from or to $^vE^{2}_{n_1+1,n_1+1}$ and $^vE^{2}_{n_2+1,n_2+1}$, hence
\begin{align*}
\left( ^vE^{\infty}_{n_1+1, n_1+1}\right)_{\bmu}&\cong \left( ^vE^{2}_{n_1+1, n+1}\right)_{\bmu}\cong  \left(H^{n_1+1}_{\mm_1}(K_{n_1+1}(\bF,C))\right)_{\bmu}\\
\quad \left( ^vE^{\infty}_{n_2+1, n_2+1}\right)_{\bmu}&\cong \left( ^vE^{2}_{n_2+1, n_2+1}\right)_{\bmu}\cong \left(H^{n_2+1}_{\mm_1}(K_{n_2+1}(\bF,C))\right)_{\bmu}.
\end{align*}

Now, if we start taking homology horizontally we obtain the same conclusions as in the proof of Theorem \ref{thm: duality} and hence the claimed assertion follows from comparing these two spectral sequences.
 \end{proof}

To conclude this section, we focus on multigraded zero-dimensional polynomial systems with coefficients in a field $\KK$, our goal being to generalize Lemma \ref{lem:SGregIsat} to the multigraded  setting. Let $f_0,\dots f_n$ be $n+1$ multihomogeneous polynomials of degree $\bd_0,\ldots, \bd_{n}$ in $R = \mathbb{K}[\bx_0,\dots, \bx_r]$. These polynomials can be considered as a specialization of the $n+1$ generic polynomials $F_0,\dots, F_{n}$. Thus, following what we did in \S \ref{section Sat and duality}, we define $I(\bff)$ as the ideal generated by $\bff:= \{f_0,\dots, f_{n}\}$, $B(\bff) := R/I(\bff)$, $B(\bff)^{\sat}:= R/ I(\bff)^{\sat}$ and $B^{\sat}(\bff):=R/I^{\sat}(\bff)$. Similarly to the single graded case, 
the graded component $B(\bff)_\bnu$, $\bnu \in \ZZ^r$, is a $\KK$-vector space and the Hilbert function of $B(\bff)$ is the function $\HF_{B(\bff)}(\bnu)=\dim_\KK B(\bff)_\bnu$. For $\bnu$ sufficient large component-wise, the Hilbert function becomes a polynomial function which is called the Hilbert polynomial and that is denoted by $\HP_{B(\bff)}(\bnu)$ (see e.g.~\cite[Proposition 4.26]{BotCh}).

\begin{prop}\label{prop:MGHFHP} 
Assume that $I(\bff)$ defines a finite subscheme in $\PP^{n_1}_{\KK}\times \ldots \times \PP^{n_r}_{\KK}$ of degree $\kappa$. Then, 
for any $\bmu \in \bdelta- (\min_{i} d_{i,1}-1,\dots ,\min_{i}d_{i,r}-1)+\NN^r$,   $$\HF_{B(\bff)^{\sat}}(\bmu)=\HF_{B^{\sat}(\bff)}(\bmu)=\HP_{B(\bff)}(\bmu)=\kappa.$$
\end{prop}
\begin{lem}\label{lem: bdelta-di, Gamma}
if $\bmu \in \bdelta- (\min_{i} d_{i,1}-1,\dots ,\min_{i}d_{i,r}-1)+
\NN^r$, then $\bmu\notin  \left(\Gamma_0\cup \Gamma_1\right)$. 
\end{lem}
\begin{proof}
We show that $\bmu\notin \Gamma_1$, the proof for $\Gamma_0$ is the same.  By the definition, one can rewrite $\Gamma_1$ as
$
\Gamma_1 = \cup_{\alpha, \lambda} 
\Gamma_1^{\alpha, \lambda}
$,
where $\Gamma_1^{\alpha, \lambda}= \sum_{j\in 
\lambda} \bd_j + Q_{\alpha}$. Fix $\alpha, \lambda$ 
and assume $i\in \alpha$. As $\sharp \lambda= n(\alpha)+1\leq 
n-1$, the $i$-th entry of every element in $
\Gamma_1^{\alpha, \lambda}$ is at most $\bdelta- d_{j,i}$ 
for all $j\notin \lambda$.
\end{proof}
\begin{proof}[Proof of Proposition \ref{prop:MGHFHP} ]
This proof follows along the same lines as the one of Lemma \ref{lem:SGregIsat}. 
	We first prove that $\HF_{B(\bff)^{\sat}}(\bmu)=\HP_{B(\bff)}(\bmu)$ in the claimed region. For that purpose, consider the spectral sequences associated with the \v{C}ech-Koszul double complex  ${C}^{\bullet}_{\bb}({K}_{\bullet}(\bff,R))$. As $I(\bff)$ defines finitely many points, if we start taking homologies horizontally the second page of the spectral sequence is of the form
$$\begin{matrix}
{\ast}&\cdots  \ast & H_2(K_\bullet) & H^0_\bb(H_1(K_\bullet)) &I(\bff)^{\sat}/I(\bff)\\
0&\cdots 0 & 0 & H^1_{\bb}(H_1(K_\bullet)) &H^1_{\bb}(B(\bff))\\
0&\cdots& 0 & 0&0\\
\vdots&\cdots&\vdots&\vdots\\
0&\cdots&0 & 0&0\\
\end{matrix}$$
where $K_\bullet$ stands for the Koszul complex $K_\bullet(\bff,R)$; recall that $H_0(K_\bullet)=B(\bff)$, $H^0_\bb(B(\bff))=I(\bff)^{\sat}/I(\bff)$ and  $H^1_{\bb}(B(\bff)) = H^1_{\bb}(B(\bff)^{\sat})$. On the other hand, the other spectral sequence is the same as in the proof of Theorem \ref{thm: duality}. Now, by definition, if $\bmu \notin \Gamma_0$ then $H^1_{\bb}(B(\bff)^{\sat})_{\bmu} \simeq H^{n+1}_\bb(K_n(\bff,R))_{\bmu}$ and setting $\bar{\bd}:= \sum_{i=0}^n \bd_i$, 
\begin{align*}
\Supp ( H^{n+1}_\bb(K_n(\bff,R))) & = \Supp (H^{n+1}_\bb(\oplus_{i=0}^{n} R(-\bar{\bd}+\bd_i)))\\
& = \Supp( \oplus _{i=0}^n H^{n+1}_\bb(R) (-\bar{\bd}+\bd_i) )\\
&= \cup_{i=0}^n (\bdelta- \bd_i -\NN^r),
\end{align*}
where the last equality follows from Example \ref{ex: top cohomology}.  By Lemma \ref{lem: bdelta-di, Gamma}, $\bmu\notin \Gamma_0$ and $\bmu \notin \Supp ( H^{n+1}_\bb(K_n(\bff,R))$, hence $H^1_{\bb}(B(\bff)^{\sat})_{\bmu}=0$. By the multigraded Grothendieck-Serre formula \cite[Proposition 4.27]{BotCh},
$$
\HP_{B(\bff)^{\sat}}(\bmu)= \HF_{B(\bff)^{\sat}}(\bmu) + \sum_{i\geq 1} (-1)^i \dim H^i_{\bb}(R/I(\bff))_\bmu.
$$
Since $H^0_{\bb}(B(\bff)^{\sat})=0$ and $H^i_{\bb}(B(\bff)^{\sat})=0$ for $i>1$, it follows that the vanishing of $H^1_{\bb}(B(\bff)^{\sat})_\bmu$  implies the expected equality $\HF_{B(\bff)^{\sat}}(\bmu)= \HP_{B(\bff)^{\sat}}(\bmu)=\HP_{B(\bff)}(\bmu)$.

Now, we turn to the proof of $\HF_{B^{\sat}(\bff)}(\bmu)=\HP_{B(\bff)}(\bmu)$ for all $\bmu$ in the claimed region. As a consequence of the Grothendieck-Serre formula again, this latter equality holds for all $\bmu$ such that  $H^0_\bb(B^{\sat}(\bff))_\bmu=0$ and $H^1_\bb(B^{\sat}(\bff))_\bmu=0$. As in Lemma \ref{lem:SGregIsat}, the vanishing of these two local cohomology modules can be controlled as fibers of projective morphisms. It turns out that \cite[Proposition 6.3]{Ch13} is stated in the classical single graded case, but a similar statement holds in the multigraded setting as it is a consequence of the more general \cite[Lemma 6.2]{Ch13}. We deduce that $H^0_\bb(B^{\sat}(\bff))_\bmu=0$ and $H^1_\bb(B^{\sat}(\bff))_\bmu=0$ for all $\bmu$ such that $H^0_\bb(B^{\sat})_\bmu=0$ and $H^1_\bb(B^{\sat})_\bmu=0$. The analysis of the two \v{C}ech-Koszul spectral sequences associated with $I$, as above, proves that 
$H^0_\bb(B^{\sat})_\bmu=0$ and $H^1_\bb(B^{\sat})_\bmu=0$ for all $\bmu \in \bdelta- (\min_{i} d_{i,1}-1,\dots, \min_{i}d_{i,r}-1)+\NN^r$, which concludes the proof.
\end{proof}

\begin{rem} As a consequence of Proposition \ref{prop:MGHFHP}, the canonical map from $I^{\sat}_\bnu$ to $I(\bff)^{\sat}_\bnu$, which is induced by the specialization $C\rightarrow R$ sending $F_i$ to $f_i$, is surjective for all $\bmu \in \bdelta- (\min_{i} d_{i,1}-1,\dots ,\min_{i}d_{i,r}-1)+\NN^r$.
\end{rem}

\section{Multigraded Sylvester forms}\label{sec:MultiGradedSylvesterForms}

In this section,  we introduce elements in $I^{\sat}$ that we call multigraded Sylvester forms. In the generic setting, these forms yield an explicit duality similar to the one described in Section \ref{subsec:SGSylvForm}. We begin with the construction of the \emph{twisted Jacobian determinant}. We use Notation \ref{Notation} and we emphasize that in the generic setting the base ring $k$ is an arbitrary commutative ring.

\subsection{Multigraded twisted Jacobian}\label{sec:JouanoloySylvForm}

As proved in Section \ref{sec:duality}, $\left(I^{\sat}/I\right)_\bdelta$ is a free $A_k$-module of rank one. A first natural task is to get an explicit generator of this module. Such a generator already appeared in the literature, notably in \cite{CCD97} and \cite{CDS98} under the name of \emph{toric Jacobian} in the more general setting of toric geometry, and in \cite{ChThesis88}, where a more algebraic treatment is proposed in the multiprojective setting. What follows in this section is strongly inspired by the construction given in \cite[Chapter III]{ChThesis88}; we provide proofs, with slight modifications, for the sake of accessibility and completeness.

\medskip

For all $i=0,\ldots, n$, we first decompose the generic polynomials $F_i$ with respect to the variables $\bx_1$ as follows
\begin{equation}\label{eq:LambdaDecomp1}
F_i=x_{1,0}F_{i,0}^{(1)}+x_{1,1}F_{i,1}^{(1)}+\cdots+x_{1,n_1}F_{i,n_1}^{(1)}.	
\end{equation}
There are many choices for such a decomposition, and we take one of them. Notice that a possible constraint to uniquely determine this decomposition is, for instance, to impose the conditions 
\begin{equation}\label{eq:unicond1}
F_{i,l}^{(1)} \in A[x_{1,l},\ldots,x_{1,n_1}][\bx_2,\dots , \bx_r], \ l=0,\ldots,n_1.
\end{equation}
Then, in a similar way, we decompose  $F_{i,n_1}^{(1)}$ with respect to $\bx_2$:
\begin{equation}\label{eq:LambdaDecomp2}
F_{i,n_1}^{(1)}=x_{2,0}F_{i,0}^{(2)}+x_{2,1}F_{i,1}^{(2)}+\cdots+x_{2,n_2}F_{i,n_2}^{(2)}.
\end{equation}
Again, there are many choices for this decomposition and we take one of them, but we can also impose, for instance, the conditions
\begin{equation}\label{eq:unicond2}
F_{i,l}^{(2)} \in A[x_{1,n_1}][x_{2,l},\ldots,x_{2,n_2}][\bx_3,\dots, \bx_r], \ l=0,\ldots,n_2.
\end{equation}
We continue this process similarly until we decompose $F_{i,n_{r-1}}^{(n_{r-1})}$ with respect to $\bx_r$. In the end, each polynomial $F_i$, $i=0,\ldots,n$, is decomposed as follows
\begin{equation}\label{eq:LambdaDecomp}
F_i=  \sum_{j=0}^{n_1-1} x_{1,j}F_{i,j}^{(1)}+ 
x_{1,n_1} \left(\sum_{j=0}^{n_2-1} x_{2,j}F_{i,j}
^{(2)}  + x_{2,n_2} \left( \cdots + x_{r-1,n_{r-1}} \left( 
\sum_{j=0}^{n_r} x_{r,j}F_{i,j}^{(r)}\right)\right)\right).		
\end{equation}
Now, we define $\Dc$ as the determinant of the following $(n+1)\times (n+1)$ matrix:
\begin{equation}\label{eq:MGMatrixD}
\left(
\begin{array}{ccccccccccc}
F_{0,0}^{(1)} & \cdots & F_{0,n_1-1}^{(1)} & F_{0,0}^{(2)} & \cdots & F_{0,n_2-1}^{(2)} & \cdots &
F_{0,n_{r-1}-1}^{(r-1)} & F_{0,0}^{(r)} & \cdots & F_{0,n_r}^{(r)} \\
\vdots & & \vdots & \vdots & & \vdots  & & \vdots & \vdots & & \vdots \\
F_{n,0}^{(1)} & \cdots & F_{n,n_1-1}^{(1)} & F_{n,0}^{(2)} & \cdots & F_{n,n_2-1}^{(2)} & \cdots &
F_{n,n_{r-1}-1}^{(r-1)} & F_{n,0}^{(r)} & \cdots & F_{n,n_r}^{(r)} 
\end{array}
\right).	
\end{equation}

\medskip

From its definition, $\Dc$ is multihomogeneous in the sets of variables $\bx_i$. More precisely, a straightforward counting shows that 
\begin{align}\label{eq:degDc}
\deg_{\bx_i} \Dc & =\delta_i-\sum_{j=i+1}^r n_j, \ i=1,\ldots,r,
\end{align}
where we recall that  $\bdelta$ is defined by \eqref{eq:delta}. As a matter of fact,  $\Dc$ is a linear form with respect to the coefficients of each polynomial $F_i$, $i=0,\ldots,n$. In addition, another property that follows directly by definition is that $\Dc$ belongs to the ideal $I^{\sat}\subset C$. Notice that the order of variables and polynomials plays an important role in the above construction, as well as the choice of decompositions.

\begin{defn}\label{def:Lambda} After a choice of ordering for polynomials and variables, and a choice of decompositions, one defines a \emph{twisted Jacobian} $\Lambda$ of the $F_i$'s by 
	$$\Lambda=\left(\prod_{i=1}^{r-1} x_{i,n_i}^{v_i}\right) \Dc,$$
	where $v_i=\sum_{j=i+1}^r n_j$ for all $i=1,\ldots,r$.
	The one uniquely determined by the choices as detailed by \eqref{eq:unicond1} and \eqref{eq:unicond2}, is denoted by $\Lambda_0$ and is called the \emph{twisted Jacobian}.
\end{defn}

 From the above definition, we deduce that $\Lambda$ is a multihomogeneous polynomial of degree $\bdelta$ and that it is a linear form with respect to the coefficients of each polynomial $F_i$.  Although $\Lambda$ is not unique, the following result shows that it is essentially unique modulo the ideal $I$.

\begin{prop}[{\cite[Theorem III.1.5]{ChThesis88}}]\label{prop:JPJ} The class of $\Lambda$ in $B=C/I$, denoted $\Delta$, is a generator of $(I^{\sat}/I)_\bdelta$ which is a free $A_{k}$-module of rank 1. In particular, $\Delta$ is independent of the order of variables and polynomials and of the choice of the decompositions \eqref{eq:LambdaDecomp} used for $\Lambda$, up to multiplication by an invertible element in $k$.  
\end{prop}

This implies the following.  Choose any order of variables and polynomials. Make the unique decomposition as in the definition of $\Lambda_0$, but following the corresponding orders in place of the one above. Then  $\Delta =\pm \Delta_0$ (the class of $\Lambda_0$). Indeed, any of these decompositions is defined for $k=\ZZ$, in which case only $\pm 1$ are invertible.\\

To prove Proposition \ref{prop:JPJ} we will need the following property.

\begin{lem}[{\cite[Lemma III.1.6]{ChThesis88}}]\label{lem:MGallcoeffs} Let $P$ be a multihomogeneous polynomial in $I^{\sat}$ of degree $\bdelta$ such that the class of $P$ in $B$ is nonzero. Then, $P$ depends on all the coefficients of all the polynomials $F_0,\ldots,F_n$.
\end{lem}
\begin{proof} As $I$ is independent on the order of polynomials $F_0,\ldots,F_n$, it is sufficient to prove the claim for the coefficients of $F_0$. Suppose that there exist $\balpha_1,\ldots,\balpha_r$, with  $|\balpha_j|=d_{0,j}$, such that $P$ does not depend on the coefficient $U_{0,\balpha_1,\ldots,\balpha_r}$ of the polynomial $F_0$. In order to emphasize this coefficient we rewrite $F_0$ as
$$F_0=  U_{0,\balpha_1,\ldots,\balpha_r} \bx_1^{\balpha_1}\bx_2^{\balpha_2}\ldots \bx_r^{\balpha_r} + \tilde{F_0}.	
$$
As $P\in I^{\sat}$, there exist $\bbeta_1,\ldots,\bbeta_r$ and polynomials $G_0,\ldots,G_n$ such that 
\begin{equation}\label{eq:MGsatequality}
\bx_1^{\bbeta_1}\bx_2^{\bbeta_2}\ldots \bx_r^{\bbeta_r}P=G_0F_0+G_1F_1+\cdots+G_nF_n \in I.
\end{equation}	
Now, in the localized ring $C_{\bx_1^{\balpha_1}\ldots \bx_r^{\balpha_r}}$ we substitute the coefficient $U_{0,\balpha_1,\ldots,\balpha_r}$ by the element $- \bx_1^{-\balpha_1}\bx_2^{-\balpha_2}\ldots \bx_r^{-\balpha_r}\tilde{F_0}$ in \eqref{eq:MGsatequality} and we deduce that 
 $$\bx_1^{\bbeta_1}\bx_2^{\bbeta_2}\ldots \bx_r^{\bbeta_r}P \in (F_1,\ldots,F_n) \subset C_{\bx_1^{\balpha_1}\ldots \bx_r^{\balpha_r}}.$$ 
It follows that there exist $\bgamma_1,\ldots,\bgamma_r$ such that 
$$\bx_1^{\bgamma_1}\bx_2^{\bgamma_2}\ldots \bx_r^{\bgamma_r}P \in (F_1,\ldots,F_n) \subset C.$$
From here, using Lemma \ref{lem: appendix} and Remark \ref{rem:TFmon}, we deduce that $P$ belongs  to $(F_1,\ldots,F_n)^{\sat}$. In addition, $P$ is of degree $\bdelta$ and since  
$$\delta_j=\sum_{i=0}^n d_{i,j}-(n_j+1)>\sum_{i=1}^n d_{i,j}-(n_j+1)$$
for all $j=1,\ldots,r$, Corollary \ref{cor: elimination ideal} implies that $(F_1,\ldots,F_n)^{\sat}_\bdelta=(F_1,\ldots,F_n)_\bdelta$. Therefore, we get a contradiction with the fact that the class of $P$ in $B$ is nonzero.
\end{proof}

\begin{proof}[Proof of Proposition \ref{prop:JPJ}.] We first prove that $\Delta\neq 0$. We proceed by induction on $n=n_1+\ldots+n_r$ (recall that $n_j\geq 1$ for all $j$). 
	
	If $n=1$ then $r=1$, $n_1=1$ and hence we are dealing with two single graded homogeneous polynomials in two variables. The claimed result hence follows from the properties of Sylvester forms in the single graded setting; see Section \ref{subsec:SGSylvForm}. 
	
	Assume $n>1$. If $\Delta=0$ then there exist multihomogeneous polynomials $G_0,\ldots,G_n$ in $C$ such that
\begin{equation}\label{eq:MGLambda1}
\Lambda = x_{1,n_1}^{v_1}x_{2,n_2}^{v_2}\ldots x_{r-1,n_{r-1}}^{v_{r-1}} \Dc = G_0F_0+G_1F_1+\cdots+G_nF_n.	
\end{equation}	 
By the construction of the determinant $\Dc$, there exist multihomogeneous polynomials $D_0,\ldots,D_n$ in $C$ such that 
\begin{equation}\label{eq:MGLambda2}
x_{1,n_1}x_{2,n_2}\ldots,x_{r,n_r} \Dc = D_0F_0+D_1F_1+\cdots+D_nF_n.	
\end{equation}
Indeed, multiplying the last column of \eqref{eq:MGMatrixD} by $x_{1,n_1}x_{2,n_2}\ldots,x_{r,n_r}$ gives a matrix whose determinant is equal to \eqref{eq:MGLambda2}. But by definition, one can add suitable multiples of the other columns of this matrix to the last one so that this last column is composed, from top to bottom, of $F_0,\ldots,F_n$. Thus, by developing this determinant with respect to the last column we get the claimed formula; for instance $D_n$ is nothing but the determinant of the top left $(n\times n)$-minor of \eqref{eq:MGMatrixD}.

Combining \eqref{eq:MGLambda1} and \eqref{eq:MGLambda2} we get that $\sum_{i=0}^n H_iF_i=0$ where for all $i=0,\ldots,n,$
$$ H_i=x_{1,n_1}^{v_1-1}x_{2,n_2}^{v_2-1}\ldots x_{r-1,n_{r-1}}^{v_{r-1}-1}D_i-x_{r,n_r}G_i \in C.$$
The polynomial $H_i$ is of degree $$(\delta_1-d_{i,1},\delta_2-d_{i,2},\ldots,\delta_{r-1}-d_{i,r-1},\delta_r-d_{i,r}+1)
=\bdelta-\bd_i+(0,\ldots,0,1).$$
We deduce that $(H_0,\ldots,H_n)$ belongs to the first syzygy module $\Syz(F_0,\ldots,F_n)$ of the polynomials $F_0,\ldots,F_n$; taking grading into account, $\Syz(F_0,\ldots,F_n)\subset \oplus_{i=0}^n C(-d_i)$ and $(H_0,\ldots,H_n)$ is a syzygy of degree $\bdelta+(0,\ldots,0,1)$. By Corollary \ref{cor: syz}, we deduce that this syzygy is a Koszul syzygy and hence  
\begin{equation}\label{eq:MGKoszulSyz}
H_n=x_{1,n_1}^{v_1-1}x_{2,n_2}^{v_2-1}\ldots x_{r-1,n_{r-1}}^{v_{r-1}-1}D_n-x_{r,n_r}G_n \in (F_0,\ldots,F_{n-1})
\end{equation} 
Consider the specialization that sends $x_{r,n_r}$ to $0$. Introducing the notation $\bar{P}=P(x_{r,n_r}=0)$ for all $P\in C$, from \eqref{eq:MGKoszulSyz} we obtain
\begin{equation}\label{eq:MGreducVar}
x_{1,n_1}^{v_1-1}x_{2,n_2}^{v_2-1}\ldots x_{r-1,n_{r-1}}^{v_{r-1}-1}\bar{D}_n \in (\bar{F}_0,\ldots,\bar{F}_{n-1}).	
\end{equation}
On the other hand, $\bar{D}_n$ is equal to the determinant $\Dc(\bar{F}_0,\ldots,\bar{F}_{n-1})$, which is constructed similarly to \eqref{eq:MGMatrixD} from the multihomogeneous polynomials $\bar{F}_0,\ldots,\bar{F}_{n-1}$ in the set of variables $\bx_1,\ldots,\bx_{r-1}$ and $(x_{r,0},\ldots,x_{r,n_r-1})$ (see the comment after \eqref{eq:MGLambda2}).  
Therefore, if $n_r\geq 2$, \eqref{eq:MGreducVar} shows that $\Delta(\bar{F}_0,\ldots,\bar{F}_{n-1})$, which is by definition the class of 
$$x_{1,n_1}^{v_1-1}x_{2,n_2}^{v_2-1}\ldots x_{r-1,n_{r-1}}^{v_{r-1}-1}\Dc(\bar{F}_0,\ldots,\bar{F}_{n-1})$$
in $C/(\bar{F}_0,\ldots,\bar{F}_{n-1})$, is equal to zero. This is in contradiction with our inductive hypothesis and hence we conclude that $\Delta\neq 0$ if $n_r\geq 2$.

If $n_r=1$ then $\bar{F}_i=x_{r,0}^{d_{i,r}}F^\flat_i(\bx_1,\ldots,\bx_{r-1})$ for all $i=0,\ldots,n-1$, where $F^\flat$ are the generic multihomogeneous polynomials in the sets of variables $\bx_1,\dots,\bx_{r-1}$. Inspecting the determinant $\bar{D}_n$, we get 
$$ \bar{D}_n = x_{r,0}^{(\sum_{i=0}^{n-1}d_{i,r})-1}\Dc(F^\flat_0,\ldots,F^\flat_{n-1}),$$
where $\Dc(F^\flat_0,\ldots,F^\flat_{n-1})$ is the determinant similar to \eqref{eq:MGMatrixD} built from the polynomials $F^\flat_0,\ldots,F^\flat_{n-1}$. Using \eqref{eq:MGreducVar}, it follows that 
\begin{equation}\label{eq:MGnr=1}
x_{1,n_1}^{n_2+\ldots+n_{r-1}}\ldots x_{r-2,n_{r-2}}^{n_{r-1}}x_{r,0}^{(\sum_{i=0}^{n-1}d_{i,r})-1}\Dc(F^\flat_0,\ldots,F^\flat_{n-1}) \in (x_{r,0}^{d_{0,r}}F^\flat_0,\ldots,x_{r,0}^{d_{n-1,r}}F^\flat_{n-1}).	
\end{equation}
But by definition \ref{def:Lambda}, $\Delta(F^\flat_0,\ldots,F^\flat_{n-1})$ is the class of 
$$x_{1,n_1}^{n_2+\ldots+n_{r-1}}\ldots x_{r-2,n_{r-2}}^{n_{r-1}} \Dc(F^\flat_0,\ldots,F^\flat_{n-1})$$
in $C/(F^\flat_0,\ldots,F^\flat_{n-1})$ and hence we deduce from \eqref{eq:MGnr=1} that 
$$x_{r,0}^{(\sum_{i=0}^{n-1}d_{i,r})-1} \Delta(F^\flat_0,\ldots,F^\flat_{n-1}) \in (x_{r,0}^{d_{0,r}}F^\flat_0,\ldots,x_{r,0}^{d_{n-1,r}}F^\flat_{n-1}),$$
which, after specializing $x_{r,0}$ to 1, shows that $\Delta(F^\flat_0,\ldots,F^\flat_{n-1})=0$, in contradiction with our inductive hypothesis. We conclude that $\Delta(F_0,\ldots,F_n)\neq 0$ if $n_r=1$, hence for all $n\geq 1$. 

\medskip

Now that we have proved that $\Delta\neq 0$, we aim to show that it is a generator of $H^0_\bb(B)_\bdelta\simeq A_k$. As the commutative ring $k$ will play an important role in what follows, we use the more precise notation $\Delta_k$. 

Let $\xi_k$ be a generator of $H^0_\bb(B_k)_\bdelta$. Thus, there exists a nonzero element $P_k \in A_k$ such that $\Delta_k=P_k\xi_k$ and we want to prove that $P_k$ is an invertible element in $A_k$. By its definition from the determinant \eqref{eq:MGMatrixD}, $\Delta_k$ is a linear form in the coefficients of each $F_i$. Since $\Delta_k=P_k\xi_k$ we deduce that $\xi_k$ is a homogeneous polynomial of degree at most 1 in the coefficients of each $F_i$, but in view of Lemma \ref{lem:MGallcoeffs}, it must be linear in the  coefficients of each $F_i$. Therefore, we deduce that $P_k\in k$. 

Let $p$ be any prime integer and set $\ZZ_p=\ZZ/p\ZZ$. By definition of $\Delta_k$ by means of a determinant, the class of $\Delta_\ZZ$ in $B_{\ZZ_p}=B_\ZZ\otimes_\ZZ \ZZ_p$ is equal to $\Delta_{\ZZ_p}$. As we have proved that  $\Delta_{\ZZ_p}\neq 0$, we deduce that the class of $P_\ZZ \xi_\ZZ$ in $B_{\ZZ_p}$ is nonzero. In particular, the class of $P_\ZZ$ in $\ZZ_p$ is nonzero. It follows that $P_\ZZ$ is an invertible element in $\ZZ$, i.e.~$P_\ZZ=\pm 1$, and hence that $\Delta_\ZZ$ is a generator of $H^0_\bb(B_\ZZ)_\bdelta$.

Now, by Lemma \ref{lem: appendix} we know that 
$$H^0_\bb(B_k)=\ker(B_k \rightarrow (B_k)_\sigma),$$
where $\sigma$ is the monomial $\prod_i x_{i,n_i}$ and the map is the canonical localization map. We deduce that we have the following commutative diagram of canonical maps where the two rows are exact:
$$\xymatrix{
0 \ar[r] & H^0_\bb(B_k)_\bdelta  \ar[r] & (B_k)_\bdelta \ar[r] & ((B_k)_{\sigma})_\bdelta  \\
         & H^0_\bb(B_\ZZ)_\bdelta\otimes_\ZZ k \ar[r] \ar[u]^{\gamma} &  (B_\ZZ)_\bdelta \otimes_\ZZ k \ar[r] \ar[u]^*[@]{\sim}  & ((B_\ZZ)_{\sigma})_\bdelta \otimes_\ZZ k \ar[u]^*[@]{\sim}.}
$$
By chasing diagram, it follows that the map $\gamma$ is surjective. But we already proved that the multiplication map 
$$ \varphi_\ZZ : A_\ZZ \rightarrow H^0_\bb(B_\ZZ)_\bdelta : Q \mapsto Q \Delta_\ZZ$$
is an isomorphism. It remains an isomorphism after tensorization by $k$ over $\ZZ$ and hence by composition with $\gamma$ we get a surjective map 
$$ \varphi_k : A_k \simeq A_\ZZ\otimes_\ZZ k \rightarrow H^0_\bb(B_k)_\bdelta : Q \mapsto Q\Delta_k.$$
We deduce that there exists an element $Q_k \in A_k$ such that $Q\Delta_k=\xi_k$. But since $\Delta_k=P_k\xi_k$ we obtain that $Q_kP_k=1$ in $A_k$, so $P_k$ is an invertible element (in $k$) and hence $\Delta_k$ is a generator of $H^0_\bb(B_k)_\bdelta$.
\medskip

To conclude, observe that the above proof applies regardless the choice of order for the variables and the polynomials, as well as the decompositions used to build the determinant $\Lambda_k$. But since we proved that the class $\Delta_k$ of $\Lambda_k$ is a generator of $H^0_\bb(B_k)_\bdelta$, we deduce that this class is independent of all these choices, up to multiplication by an invertible element in $k$.
\end{proof}

Before closing this section, we explain why the determinant $\Lambda$ is called a twisted Jacobian determinant. Partial derivatives and the Euler formula can be used to get decompositions similar to the ones used to define the determinant $\Dc$. Indeed, pick an integer $i\in\{0,\ldots,n\}$, by the Euler formula 
$$ d_{i,1}F_i=x_{1,0}\frac{\partial{F_i}}{\partial x_{1,0}}+x_{1,1}\frac{\partial{F_i}}{\partial x_{1,1}}+\cdots+x_{1,n_1}\frac{\partial{F_i}}{\partial x_{1,n_1}},$$
which provides a decomposition similar to \eqref{eq:LambdaDecomp1} with the difference that we need to multiply $F_i$ by $d_{i,1}$. Applying Euler formula to $\frac{\partial{F_i}}{\partial x_{1,n_1}}$ with respect to the variables $\bx_2$ implies
$$d_{i,2} \frac{\partial F_i}{\partial x_{1,n_1}} =
x_{2,0} \frac{\partial^2 F_i}{\partial x_{1,n_1}\partial x_{2,0}}+
x_{2,1} \frac{\partial^2 F_i}{\partial x_{1,n_1}\partial x_{2,1}}+\cdots+
x_{2,n_2} \frac{\partial^2 F_i}{\partial x_{1,n_1}\partial x_{2,n_2}},$$
which is very similar to \eqref{eq:LambdaDecomp2}. Continuing this way one can build a determinant similar to $\Dc$ where the entries are replaced by partial derivatives. We denote by $\Jc(F_0,\ldots,F_n)$ this Jacobian determinant. Comparing the degrees, we see that $\Jc$ and $\Dc$ have the same degree with respect to each set of variables $\bx_i$ and each set of coefficients of any polynomial $F_j$. Thus, denoting by $J(F_0,\ldots,F_n)$ the class of 
$$\left(\prod_{i=1}^{r-1} x_{i,n_i}^{v_i}\right) \Jc(F_0,\ldots,F_n)$$
in $(B_k)_\bdelta$, $J$ and $\Delta$ are expected to differ by a multiplicative element in $k$. Actually, one can show that \cite[Proposition III.2.6]{ChThesis88}
	$$J(F_0,\ldots,F_n)=\left(\prod_{\substack{ 0\leq i \leq n \\ 1\leq k \leq r}} d_{i,j} \right) \Delta_0 (F_0,\ldots,F_n)$$
in $(B_\ZZ)_\bdelta\simeq A_\ZZ$. So, the twisted Jacobian $\Delta_0$ provides a generator of $H^0_\bb(B_k)_\bdelta$ for any commutative ring $k$, whereas the Jacobian $J$ is sensitive to finite characteristic settings.

\subsection{Multigraded Sylvester forms}\label{sec:MultiSylvForms}
In this section, we introduce multigraded Sylvester forms which are generalizations of the twisted Jacobian determinant. These forms provide additional nonzero elements in $I^{\sat}/I$ of degree lower than $\bdelta$ (component-wise) and generate  some  graded components under suitable assumptions.

\medskip

For all $j\in \{1,\ldots,r\}$,  choose a multi-index of non negative integers $\balpha_j=(\alpha^{(j)}_0,\ldots,\alpha^{(j)}_{n_j})$ such that 
$$|\balpha_j|=\sum_{i=0}^n \alpha^{(j)}_i < \min_{i\in\{0,\ldots,n\}} d_{i,j}.$$ 
Under these assumptions one can always decompose each polynomial $F_i$ as follows: 
\begin{multline}\label{eq:decompSylvFi}
F_i=  \sum_{j=0}^{n_1-1} x_{1,j}^{\alpha^{(1)}_{j}+1}F_{i,j}^{(1)}+ 
x_{1,n_1}^{\alpha^{(1)}_{n_1}+1} \left(\sum_{j=0}^{n_2-1} x_{2,j}^{\alpha^{(2)}_{j}+1}F_{i,j}
^{(2)} \right. \\
\left. + x_{2,n_2}^{\alpha^{(2)}_{n_2}+1} \left( \cdots + x_{r-1,n_{r-1}}^{\alpha^{(r-1)}_{n_{r-1}}+1} \left( 
\sum_{j=0}^{n_r} x_{r,j}^{\alpha^{(r)}_{j}+1}F_{i,j}^{(r)}\right)\right)\right)
\end{multline}
where $F_{i,j}^{(l)}$ are multihomogeneous polynomials.
Define $\Dc_{\balpha_1,\dots,\balpha_r}$ as the determinant of the following matrix:
\begin{equation}\label{eq:MGDalphaDet}
\left(
\begin{array}{ccccccccccc}
F_{0,0}^{(1)} & \cdots & F_{0,n_1-1}^{(1)} & F_{0,0}^{(2)} & \cdots & F_{0,n_2-1}^{(2)} & \cdots & F_{0,n_{r-1}-1}^{(r-1)} & F_{0,0}^{(r)} & \cdots & F_{0,n_r}^{(r)} \\
\vdots & & \vdots & \vdots & & \vdots & & \vdots & \vdots & & \vdots \\
F_{n,0}^{(1)} & \cdots & F_{n,n_1-1}^{(1)} & F_{n,0}^{(2)} & \cdots & F_{n,n_2-1}^{(2)} & \cdots & F_{n,n_{r-1}-1}^{(r-1)} & F_{n,0}^{(r)} & \cdots & F_{n,n_r}^{(r)}
\end{array}\right).	
\end{equation}
\medskip

From its definition, it is straightforward to check that $\Dc_{\balpha_1,\dots, \balpha_r}$ is a multihomogeneous polynomial. For all $i=0,\ldots,r-1$,
$$\deg_{\bx_i}(\Dc_{\balpha_1,\dots,\balpha_r})=\delta_i - \vert \balpha_i\vert - (\alpha^{(i)}_{n_i}+1)(n_{i+1}+\cdots+ n_r)=\delta_i - \vert \balpha_i\vert - (\alpha^{(i)}_{n_i}+1)v_i,$$
and $$\deg_{\bx_r}( \Dc_{\balpha_1,\dots, \balpha_r})=\delta_r - \vert \balpha_r\vert.$$
Observe that $\Dc_{\balpha_1\dots,\balpha_r}$ is a linear form in the coefficients of each polynomial $F_i$. In addition, it is immediate to verify that $\Dc_{\balpha_1,\dots,\balpha_r}$ belongs to $I^{\sat}$ (as a consequence of Lemma \ref{lem: appendix}, it is sufficient to check that $\Dc_{\balpha_1\dots,\balpha_r}$ multiplied by a certain monomial is in $I$).

\begin{defn}\label{def:MGSylvForm}  After a choice of ordering for polynomials and variables, and a choice of decomposition, the Sylvester form of degree $\balpha_1,\dots,\balpha_r$ is defined as
$$ \Sylv_{\balpha_1,\dots,\balpha_r}:=
\left(\prod_{i=1}^{r-1} x_{i,n_i}^{(\alpha^{(i)}_{n_i}+1)v_i}\right) \Dc_{\balpha_1,\dots,\balpha_r} \in A_k[\bx_1,\dots, \bx_r].$$	
The class of $\Sylv_{\balpha_1,\dots,\balpha_r}$ in $B=C/I$ is denoted by $\sylv_{\balpha_1,\dots,\balpha_r}$ and we have 
$$\sylv_{\balpha_1,\dots,\balpha_r} \in (I^{\sat}/I)_{(\delta_1-|\balpha_1|,\dots , \delta_r-|\balpha_r|)}.$$
\end{defn}

Observe that this definition generalizes Definition \ref{def:Lambda} since $\Sylv_{(\b0,\ldots,\b0)}$, respectively $\sylv_{(\b0,\ldots,\b0)}$, is nothing but the twisted Jacobian $\Lambda$, respectively $\Delta$. The next results aim to generalize Proposition \ref{prop:JPJ} that shows that $\sylv_{(\b0,\ldots,\b0)}$ is a generator of the free   $A_k$-module  $(I^{\sat}/I)_\bdelta$ of rank one.

\begin{thm} \label{prop multiplication}
Let $\balpha_1,\dots \balpha_r$, and $\bbeta_1,\dots,\bbeta_r$ be multi-indices of non negative integers such that $|\balpha_i|=|\bbeta_i|$ for all $i=1,\ldots,r$. Then 
$$ \left(\prod_{i=1}^r\bx_i^{\bbeta_i}\right) \sylv_{\balpha_1,\ldots, \balpha_r}=
\begin{cases}
	0 \textrm{ if } (\balpha_1,\dots,\balpha_r) \neq   (\bbeta_1,\dots, \bbeta_r), \\
	\sylv_{\b0,\ldots, \b0} \ \ \mathrm{ otherwise. }
\end{cases}
$$
\end{thm}
\begin{proof} We begin with some observations about the determinant $\Dc_{\balpha_1,\ldots,\balpha_r}$ defined by \eqref{eq:MGDalphaDet}. Multiply the first column of \eqref{eq:MGDalphaDet} by $x_{1,0}^{\alpha_0^{(1)+1}}$ and  add suitable multiples of the other columns  according to the decomposition \eqref{eq:decompSylvFi}, the first column would become the column vector of $F_0,\ldots,F_n$. Therefore
$$x_{1,0}^{\alpha_0^{(1)+1}}\Dc_{\balpha_1,\ldots,\balpha_r} \in I.$$
With the same argument, 
\begin{align}\label{eq:MGrelationDc}
	\nonumber
 	x_{1,j}^{\alpha_j^{(1)}+1}\Dc_{\balpha_1,\ldots,\balpha_r} & \in I, \ j=0,\ldots,n_1-1,\\ 	\nonumber
	x_{1,n_1}^{\alpha_{n_1}^{(1)}+1} x_{2,j}^{\alpha_j^{(2)}+1}\Dc_{\balpha_1,\ldots,\balpha_r} & \in I, \ j=0,\ldots,n_2-1,\\ 	\nonumber
	\vdots & \\ 	\nonumber
	x_{1,n_1}^{\alpha_{n_1}^{(1)}+1}\ldots x_{r-2,n_{r-2}}^{\alpha_{n_{r-2}}^{(r-2)}+1}
	x_{r-1,j}^{\alpha_{j}^{(r-1)}+1}\Dc_{\balpha_1,\ldots,\balpha_r} & \in I, \ j=0,\ldots,n_{r-1}-1,\\ 
	x_{1,n_1}^{\alpha_{n_1}^{(1)}+1}\ldots x_{r-2,n_{r-2}}^{\alpha_{n_{r-2}}^{(r-2)}+1}
	x_{r-1,n_{r-1}}^{\alpha_{n_{r-1}}^{(r-1)}+1}x_{r,j}^{\alpha_{j}^{(r)}+1}\Dc_{\balpha_1,\ldots,\balpha_r} & \in I, \ j=0,\ldots,n_{r}.
\end{align}
By Definition \ref{def:MGSylvForm}
$$\bx_1^{\bbeta_1}\ldots\bx_r^{\bbeta_r}\Sylv_{\balpha_1,\ldots,\balpha_r}=\bx_1^{\bbeta_1}\ldots\bx_r^{\bbeta_r}
x_{1,n_1}^{(\alpha^{(1)}_{n_1}+1)v_1} \ldots x_{r-1,n_{r-1}}^{(\alpha^{(r-1)}_{n_{r-1}}+1)v_{r-1}}
 \Dc_{\balpha_1,\dots,\balpha_r}$$
where $v_i$ is a positive integer for all $i=1,\ldots,r-1$. Using \eqref{eq:MGrelationDc}, we deduce that if $\beta_j^{(r)}>\alpha_j^{(r)}$ from some $j\in\{0,\ldots,n_r\}$ then $\bx_1^{\bbeta_1}\ldots\bx_r^{\bbeta_r}\sylv_{\balpha_1,\ldots,\balpha_r}=0$. But since $|\bbeta_r|=|\balpha_r|$ by the assumption, this condition is equivalent to $\bbeta_r\neq \balpha_r$. 
In addition, since the class $\sylv_{\balpha_1,\ldots,\balpha_r}$ is independent on the order of variables chosen to build the determinant $\Dc_{\balpha_1,\ldots,\balpha_r}$ (see Proposition \ref{prop:JPJ}), one can show by a similar argument that if $\bbeta_j\neq \balpha_j$ for some $j\in\{1,\ldots,r\}$, i.e.~if $(\balpha_1,\ldots,\balpha_r)\neq(\bbeta_1,\ldots,\bbeta_r)$, then $\bx_1^{\bbeta_1}\ldots\bx_r^{\bbeta_r}\sylv_{\balpha_1,\ldots,\balpha_r}=0$.

To conclude the proof, it remains to show that $$\bx_1^{\balpha_1}\ldots\bx_r^{\balpha_r}\sylv_{\balpha_1,\ldots,\balpha_r}=\sylv_{(\b0,\ldots,\b0)}.$$
For that purpose, starting from decompositions of the form \eqref{eq:decompSylvFi} to build the determinant $\Dc_{\balpha_1,\ldots,\balpha_r}$, observe that one can multiply the polynomials $F_{i,j}^{l}$ by suitable monomials to get decompositions of the form \eqref{eq:LambdaDecomp} that are used to build the determinant $\Dc_{(\b0,\ldots,\b0)}$. In this way, using appropriate decompositions, we identify that 
$$ \Dc_{(\b0,\ldots,\b0)}=\bx_1^{\balpha_1} x_{1,n_1}^{v_1\alpha_{n_1}^{(1)}}
\bx_2^{\balpha_2} x_{2,n_2}^{v_2\alpha_{n_2}^{(2)}}\ldots \bx_{r-1}^{\balpha_{r-1}} x_{r-1,n_{r-1}}^{v_{r-1}\alpha_{n_{r-1}}^{(r-1)}}\bx_r^{\balpha_r}\Dc_{\balpha_1,\ldots,\balpha_r}.$$
By Definition \eqref{def:Lambda} and Definition \eqref{def:MGSylvForm} we deduce that
\begin{align*}
	\sylv_{(\b0,\ldots,\b0)} &= x_{1,n_1}^{v_1}\ldots x_{r-1,n_{r-1}}^{v_{r-1}} \Dc_{(\b0,\ldots,\b0)} \\
	 &= \bx_1^{\balpha_1} x_{1,n_1}^{v_1\alpha_{n_1}^{(1)}+v_1}
\ldots \bx_{r-1}^{\balpha_{r-1}} x_{r-1,n_{r-1}}^{v_{r-1}\alpha_{n_{r-1}}^{(r-1)}+v_{r-1}}\bx_r^{\balpha_r}\Dc_{\balpha_1,\ldots,\balpha_r}\\
    &=\bx_1^{\balpha_1}\ldots\bx_r^{\balpha_r}\sylv_{\balpha_1,\ldots,\balpha_r},
\end{align*}
as claimed.
\end{proof}

\begin{rem}\label{rem:ref} As pointed out by an anonymous reviewer, in the case $k$ is the field of complex numbers, the above result can be deduced from \cite{CCD97}, as a consequence of the Global Transformation Law \cite[Theorem 0.1]{CCD97}, which encapsulates  Sylvester-like decompositions, and \cite[Proposition 2.4]{CCD97}.
\end{rem}

\begin{thm}\label{thm:main} For all $j=1,\ldots,r$ let $\mu_j$ be an integer such that $0\leq \mu_j <  \min_i d_{i,j}$ and set $\bmu=(\mu_1,\ldots,\mu_r)$. Then, 
the set of multigraded Sylvester forms
$$\left\{\sylv_{\balpha_1,\ldots,\balpha_r} \right\}_{|\balpha_j|=\mu_j, \ j=1,\ldots,r }$$
yields an $A_k$-basis of the free $A_k$-module $(I^{\sat}/I)_{\bdelta-\bmu}$.
\end{thm}

\begin{proof}
Since  $n_i\geq 1$ for every $1\leq i\leq r$, every degree in $\Gamma_0\cup \Gamma_1$ has at least one coordinate, namely $\ell$,  which is greater than or equal to $\min_j d_{j,k}$.  Hence, in this case, $\boldsymbol{\mu}\notin 
\Gamma_0\cup \Gamma_1$, therefore by Theorem \ref{thm: duality}, 
$$H^0_\bb(B)_{\bdelta- \boldsymbol{\mu}} \xrightarrow{\sim} \mathrm{Hom}_A(C_{\boldsymbol{\mu}},H^0_\bb(B)_{\delta})
\xrightarrow{\sim} {\check{C}}_{\boldsymbol{\mu}}.$$
	By definition,
	$${\check{C}}_{\boldsymbol{\mu}}:= \mathrm{Hom}_A(C_{\boldsymbol{\mu}},A)$$ 
	and the canonical $A$-basis of this free $A$-module is identified with the multihomogeneous monomials in the sets of variables  $\bx_1,\ldots, \bx_r$ of degree $\bmu$:
	\begin{eqnarray*}
		(\bx_1^{\balpha_1}\ldots \bx_r^{\balpha_r})^{\check{ }} : C_{\bmu} & \rightarrow & A \\
		(\bx_1^{\balpha_1}\ldots \bx_r^{\balpha_r})& \mapsto & 1 \\
		(\bx_1^{\balpha'_1}\ldots \bx_r^{\balpha'_r}) & \mapsto & 0 \textrm{ if } (\balpha'_1,\ldots, \balpha'_r)\neq (\balpha_1,\ldots, \balpha_r).
	\end{eqnarray*}
	Therefore, by Theorem \ref{thm: duality} 
	$${\check{C}}_{\bmu} \simeq \mathrm{Hom}_A(C_{\bmu},H^0_\bb(B)_{\bdelta})$$ 
	and an $A$-basis of $\mathrm{Hom}_A(C_{\bmu},H^0_\bb(B)_{\bdelta})$ is given by $\phi \circ (\bx_1^{\balpha_1}\cdots \bx_r^{\balpha_r})^{\check{ }}$, $|\balpha_1|=\mu_1,\ldots, |\balpha|=\mu_r$. By using Proposition \ref{prop multiplication}, it turns out that the map $\phi \circ (\bx_1^{\balpha_1}\ldots \bx_r^{\balpha_r})^{\check{ }}$ is nothing but the multiplication map by $\sylv_{\balpha_1,\dots,\balpha_r}$. 
	
	Now, by the argument using spectral sequences, we have an isomorphism of $A$-modules:
	$$ H^0_\bb(B)_{\bdelta-\bmu} \xrightarrow{\sim}  \mathrm{Hom}_A(C_{\bmu},H^0_\bb(B)_{\bdelta}).$$
	The canonical map that sends $\sylv_{\balpha_1,\dots,\balpha_r} \in H^0_\bb(B)_{\bdelta- \bmu}$ to the multiplication map by $\sylv_{\balpha_1,\dots,\balpha_r}$ in $\mathrm{Hom}_A(C_{\bmu},H^0_\bb(B)_{\bdelta})$ realizes the above isomorphism.	
\end{proof}

\begin{rem}
As a consequence of the above theorem, the classes of Sylvester forms $\sylv_{\balpha_1,\dots, \balpha_{r}}$ are independent of the choice of decompositions \eqref{eq:decompSylvFi}, up to multiplication by an invertible element in $k$, that is independent of ${\balpha_1,\dots, \balpha_{r}}$. This element is $\pm 1$ if one restricts to the unique decompositions as in the beginning of this section, varying orders of variables and forms (more generally to any decomposition that lifts to the base ring $\ZZ$).
\end{rem}

\section{Application to Multigraded Elimination Matrices}\label{sec:MGelimMat}

In this section, the results obtained in Section \ref{sec:duality} and Section \ref{sec:MultiGradedSylvesterForms} are applied to build a family of elimination matrices for multihomogeneous polynomial systems.

\subsection{Hybrid elimination matrices}
Adopt Notation \ref{Notation} and define a family of matrices indexed by $\bnu \in \NN^r$ as follows. First, for all $\bnu \in \NN^r$ such that 
\begin{equation}\label{eq=nucond1}
\bnu \notin (\bdelta-\N^r)\cup \Gamma_0\cup \Gamma_1
\end{equation}
we define the matrix $\MM_\bnu$ as the matrix of the $A$-linear map
\begin{eqnarray}\label{eq:MGMacMat}
\oplus_{i=0}^n C_{\bnu-\bd_i}  & \rightarrow & C_\bnu \\ \nonumber
(G_0,\ldots,G_n) & \mapsto & \sum_{i=0}^n G_iF_i 
\end{eqnarray}
in canonical bases. This matrix is the classical multigraded Macaulay-type matrix of $F_0,\ldots,F_n$ in degree $\bnu$. Second, for all 
$\bnu \in \NN^r$ such that 
\begin{equation}\label{eq=nucond2}
\bnu=\bdelta-(\mu_1,\ldots\mu_r) \textrm{ such that } 0\leq \mu_j<\min_i d_{i,j} \textrm{ for all } j=1,\ldots,r,
\end{equation}
we define the matrix $\HH_\bnu$ as the matrix of the $A$-linear map
\begin{eqnarray}\label{eq:MGHybMat}
\oplus_{i=0}^n C_{\bnu-d_i}  \oplus_{\balpha : |\alpha_j|=\delta_j-\nu_j} A & \rightarrow & C_\bnu \\ \nonumber
(G_0,\ldots,G_n, \ldots,\ell_\balpha,\ldots) & \mapsto & \sum_{i=0}^n G_iF_i + \sum_{\balpha : |\alpha_j|=\delta_j-\nu_j}\ell_\balpha \Sylv_{\balpha}
\end{eqnarray}
in canonical bases. This matrix in an hybrid matrix: it has a Macaulay-type block and another block built from multigraded Sylvester forms of degree $\bnu$. 

We notice that the smallest matrix in the entire family made of the matrices $\HH_\bnu$ and $\MM_\bnu$ we have just defined is obtained for $\bnu=\bdelta-(\min_i d_{i,1}-1,\ldots,\min_i d_{i,r}-1)$. 

\medskip

Given a polynomial system $\bff:=\{f_0,\dots f_n\}$ of $n+1$ multihomogeneous polynomials of degree $\bd_1,\dots, \bd_{n+1}$ in $R = \mathbb{K}[\bx_0,\dots, \bx_r]$, we recall that the notation $I(\bff)$, $B(\bff)$, $B^{\sat}(\bff)$ and $\MM_\bnu(\bff)$ respectively, stand for the specialization of $I, B, B^{\sat}$ and $\MM_\nu$, respectively.

\begin{prop} If $\bnu\in \NN^r$ satisfies \eqref{eq=nucond1}, respectively \eqref{eq=nucond2}, then $\MM_\bnu$, respectively $\HH_\nu$, is a presentation matrix of the $A$-module $(B^{\sat})_\bnu$. 
	In particular, $\MM_\nu(\bff)$, respectively  $\HH_\nu(\bff)$, is surjective if and only if $\bff$ has no common zeroes in  $\PP_{\bar{\KK}}^{n_1}\times \cdots \times \PP_{\bar{\KK}}^{n_r}$. 	
\end{prop}
\begin{proof} If $\bnu$ satisfies \eqref{eq=nucond1} this result follows from Theorem \ref{thm: duality} and if $\bnu$ satisfies \eqref{eq=nucond2} it follows from Theorem \ref{thm:main}.
\end{proof}

\subsection{The drop-of-rank property} As already mentioned, a key feature of elimination matrices for solving polynomial systems with coefficients over a field is to have the drop-of-rank property. It turns out that the family of matrices $\MM_\bnu$ defined above has this property. 

To be more precise, consider a multigraded zero-dimensional polynomial system with coefficients in a field $\mathbb{K}$: $\bff:= \{f_0,\dots, f_{n}\}$ are $n+1$ multihomogeneous polynomials of degree $\bd_1,\dots, \bd_{n+1}$ in $R = \mathbb{K}[\bx_0,\dots, \bx_r]$. These polynomials are specialization of the $n+1$ generic polynomials $F_0,\dots, F_{n}$. 

\begin{prop}\label{prop:MGHybMat} Assume  $\bff$ defines a finite subscheme in $\PP_{\bar{\KK}}^{n_1}\times \cdots \times \PP_{\bar{\KK}}^{n_r}$ of degree $\kappa$. If
\begin{equation}\label{eq=cond3}
\bnu \in \bdelta- (\min_{i} d_{i,1}-1,\ldots ,\min_{i}d_{i,r}-1)+\NN^r,	
\end{equation}	
the corank of $\MM_\bnu(\bff)$, or $\HH_\nu(\bff)$ depending on $\bnu$, is equal to $\kappa$.
\end{prop}
\begin{proof} Let $\bnu$ be a multi-index satisfying \eqref{eq=cond3}, by Lemma \ref{lem: bdelta-di, Gamma}, $\bnu$ satisfies \eqref{eq=nucond1} or \eqref{eq=nucond2}. Therefore, the claimed result follows straightforwardly from Proposition \ref{prop:MGHFHP} as $\MM_\bnu(\bff)$, or $\HH_\nu(\bff)$ depending on $\bnu$, yields a presentation matrix of the $\KK$-vector space $(B^{\sat}(\bff))_\bnu$.
\end{proof}

To conclude, we provide two illustrative examples.

\begin{exmp}\label{example1}
Consider the case of the $n+1=5$ generic bihomogeneous  polynomials of degree $\bd= (3,3)$ over $\PP^2\times \PP^2$ ($r=2$, $n_1=n_2=2$). Form definitions, $\bdelta = (12,12)$ and 
\begin{align*}
\Gamma_0 &= \left((3,6)+ (-\N, \N)\right) \cup\left( (6,3)+ (\N, -\N)\right), \\
\Gamma_1 &= \left((6,9)+ (-\N, \N)\right) \cup \left((9,6)+ (\N, -\N)\right),\\
\Gamma_2 &= \left((9,12)+ (-\N, \N) \right) \cup \left((12,9)+ (\N, -\N)\right).
\end{align*}

In this case, the determinant $\Dc$ introduced in Section \ref{sec:JouanoloySylvForm} has degree $(12-2,12)= (10,12)$ and by Proposition \ref{prop:JPJ}, the twisted Jacobian determinant $\Delta:= x_{1,2}^2\cdot \Dc$ is a generator of $H^0_{\bb}(B)_\bdelta$. More generally, let $\balpha_1=(\alpha_0^{(1)},\alpha_1^{(1)},\alpha_2^{(1)})$ and $\balpha_2=(\alpha_0^{(2)},\alpha_1^{(2)},\alpha_2^{(2)})$ be multi-indices such that $0\leq |\balpha_i| \leq 2$. By Theorem \ref{thm:main}, the Sylvester forms $\sylv_{\balpha_1,\balpha_2}$ introduced in Section \ref{sec:MultiSylvForms} yield bases of $H^0_{\bb}(B)_{\bnu}$ with $\bnu=(12-|\balpha_1|,12-|\balpha_2|)$.

The matrix $\HH_{(12,12)}$ is the matrix of a map of the form $C_{(9,9)}^3\oplus A \rightarrow C_{(12,12)}$ and it is of size $15126\times 8281$. It is a Macaulay-type matrix except for one column which is filled with the coefficient of the twisted Jacobian determinant. Such a matrix already appeared in \cite[Proposition 2.1]{CDS98}. From Proposition \ref{prop:MGHybMat}, the smallest elimination matrix having the drop-of-rank property we get is $\HH_{(10,10)}$; it is the matrix of a map of the form $ C_{(7,7)}^3\oplus A^{36} \rightarrow C_{(10,10)}$ and it is of size $6516\times 4356$.
\end{exmp}

\begin{exmp}[Dixon resultant matrices]\label{exmp: Dixon} Dixon in \cite{Dixon} describes three determinantal formulas for computing the resultant of three generic bihomogeneous polynomials of the same degree $(m,n)$ over $\PP^1\times \PP^1$. These determinants are of order $6mn,4mn$, and $2mn$ and the entries of the corresponding  matrices are respectively homogeneous of degree $1,2$, and $3$ in the coefficients. It turns out that the elimination matrices $\MM_\bnu$ we have defined include the Dixon determinants of order $6mn$ and $4mn$, as well as some other intermediate matrices that already appeared in \cite[Section 5]{Goldman}. 
	
To be more precise, we consider the case the $n+1=3$ generic bihomogeneous polynomials over $\PP^1\times \PP^1$ ($r=2$, $n_1=n_2=1$). 
From definition, $\bdelta=(3m-2,3n-2)$, and 
\begin{align*}
\Gamma_0 &= (m-2,n)+ (-\N, \N) \cup (m,n-2)+ (\N, -\N), \\
\Gamma_1 &= (2m-2,2n)+ (-\N, \N) \cup (2m,2n-2)+ (\N, -\N),\\
\Gamma_2 &= (3m-2,3n)+ (-\N, \N) \cup (3m,3n-2)+ (\N, -\N).
\end{align*}
As $\Gamma_0 \subset (\bdelta-\NN^r)\cup \Gamma_1$, Corollary \ref{cor- main theorem} implies that $(I^{\sat}/I)\bnu=0$ for any $\bnu \notin (\bdelta-\NN^r)\cup \Gamma_1$. In the following picture we plot the regions $(\bdelta-\NN^r)$ in grey and $\Gamma_1$ in yellow. 

\medskip

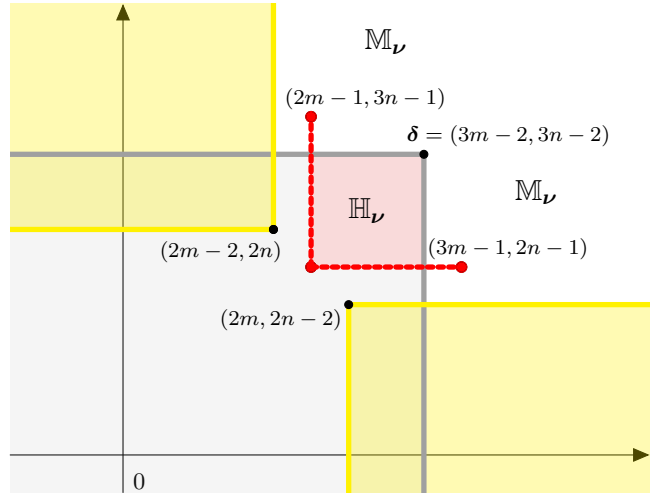
\begin{center}
\definecolor{aqaqaq}{rgb}{0.6274509803921569,0.6274509803921569,0.6274509803921569}
\definecolor{cqcqcq}{rgb}{0.7529411764705882,0.7529411764705882,0.7529411764705882}
\definecolor{ffqqqq}{rgb}{1.,0.,0.}
\begin{tikzpicture}[line cap=round,line join=round,>=triangle 45,x=1cm,y=1cm]
\draw[->,color=black] (-1.5,0.) -- (7,0.);
\draw[->,color=black] (0.,-0.5) -- (0.,6);
\draw[color=black] (0pt,-10pt) node[right] {\footnotesize $0$};
\clip(-1.5,-0.5) rectangle (7,6);
\fill[line width=2.pt,color=aqaqaq,fill=aqaqaq,fill opacity=0.10000000149011612] (4.,4.) -- (4.,-3.) -- (-2.,-2.) -- (-2.,4.) -- cycle;
\fill[line width=2.pt,color=yellow,fill=yellow,fill opacity=0.30000000149011612] (2.,3.) -- (2.,10.) -- (-2.,10.) -- (-2.,3.) -- cycle;
\fill[line width=2.pt,color=yellow,fill=yellow,fill opacity=0.30000000149011612] (3.,2.) -- (3.,-2.) -- (10.,-2.) -- (10.,2.) -- cycle;

\fill[line width=2.pt,color=pink,fill=pink,fill opacity=0.5] (2.5,4) -- (2.5,2.5) -- (4.,2.5) -- (4,4) -- cycle;

\draw [line width=2.pt,color=aqaqaq] (4.,4.)-- (4.,-3.);
\draw [line width=2.pt,color=aqaqaq] (4.,-3.)-- (-2.,-2.);
\draw [line width=2.pt,color=aqaqaq] (-2.,-2.)-- (-2.,4.);
\draw [line width=2.pt,color=aqaqaq] (-2.,4.)-- (4.,4.);
\draw [line width=2.pt,color=yellow] (2.,3.)-- (2.,10.);
\draw [line width=2.pt,color=aqaqaq] (2.,10.)-- (-2.,10.);
\draw [line width=2.pt,color=aqaqaq] (-2.,10.)-- (-2.,3.);
\draw [line width=2.pt,color=yellow] (-2.,3.)-- (2.,3.);
\draw [line width=2.pt,color=yellow] (3.,2.)-- (3.,-2.);
\draw [line width=2.pt,color=aqaqaq] (3.,-2.)-- (10.,-2.);
\draw [line width=2.pt,color=aqaqaq] (10.,-2.)-- (10.,2.);
\draw [line width=2.pt,color=yellow] (10.,2.)-- (3.,2.);
\begin{scriptsize}
\draw [fill=black] (4.,4.) circle (1.5pt);
\draw[color=black] (5.123952757134727,4.248232930089371) node {$\bdelta = (3m-2,3n-2)$};
\draw [fill=black] (4.,-3.) circle (1.5pt);
\draw [fill=black] (-2.,-2.) circle (1.5pt);

\draw [fill=black] (2.,3.) circle (1.5pt);
\draw[color=black] (1.3,2.7) node {$(2m-2,2n)$};

\draw [fill=black] (2.,3.) circle (1.5pt);
\draw [fill=black] (3.,2.) circle (1.5pt);
\draw [fill=ffqqqq] (2.5,2.5) circle (2pt);
\draw [color=ffqqqq] (2.5,2.5) circle (2pt);
\draw[color=black] (2.1,1.8) node {$(2m,2n-2)$};
\draw[color=black] (3.220808870431704,4.740491540665641) node {$(2m-1,3n-1)$};
\draw [fill=ffqqqq] (4.5,2.5) circle (2pt);
\draw [color=ffqqqq] (4.5,2.5) circle (2pt);
\draw [fill=ffqqqq] (2.5,4.5) circle (2pt);
\draw [color=ffqqqq] (2.5,4.5) circle (2pt);

\draw[color=black] (5.116211367710998,2.74050890433863536) node {$(3m-1,2n-1)$};
\draw [line width=2.pt,dash pattern=on 2pt off 2pt,color=ffqqqq] (4.4,2.5)-- (2.5,2.5);
\draw [line width=2.pt,dash pattern=on 2pt off 2pt,color=ffqqqq] (2.5,4.4)-- (2.5,2.5);

\end{scriptsize}
\draw[color=black] (3.25,3.25) node {$\HH_{\bnu}$};
\draw[color=black] (3.5,5.5) node {$\MM_{\bnu}$};
\draw[color=black] (5.5,3.5) node {$\MM_{\bnu}$};
\end{tikzpicture}
\end{center}

\medskip

\noindent By Proposition \ref{prop:MGHybMat}, the matrices $\HH_\bnu$ and $\MM_{\bnu}$ such that $\bnu=(\nu_1,\nu_2)$ with $\nu_1\geq 2m-1$ and $\nu_2\geq 2n-1$ are all elimination matrices that have the drop-of-rank property. Those in the white area are purely of Macaulay type whereas those in the pink area involve Sylvester forms. Among these matrices, we can identify the following Dixon determinants that are located on the red dotted segments in the above picture:
\begin{itemize}
	\item $\MM_{(2m-1,3n-1)}$ and $\MM_{(2m-1,3n-1)}$ are square matrices associated to maps of the form  $C_{(m-1, 2n-1)}^3\rightarrow C_{(2m-1,3n-1)}$. They both correspond to the Dixon determinants of order $6mn$. 
	\item $\HH_{2m-1,2n-1}$ is a square matrix associated with a map of the form $C_{(m-1,n-1)}^3\oplus A^{mn}\rightarrow C_{(2m-1,2n-1)}$. It corresponds to the Dixon determinant of order $4mn$.
	\item Let $i$ be an integer such that $0< i < m$. The matrix $\HH_{(2m-1+i,2n-1)}$ is associated with the map
	$$
	C_{(m-1+i,n-1)}^3\oplus A^{|\mathcal{I}|}\rightarrow C_{(2m-1+i,2n-1)},
	$$
	where $\mathcal{I}$ is the set of Sylvester forms of degree $(2m-1+i,2n-1)$. By Theorem \ref{thm: duality},   $|\mathcal{I}|= (m-i)n$ as it is equal to the number of monomials of degree $\bdelta - (2m-1+i,2n-1)$. Therefore, 
 $\HH_{(2m-1+i,2n-1)}$ is a square matrix. Similarly, the same conclusion holds for the matrices $\HH_{(2m-1,2n-1+j)}$ with $0<j<n$. Thus, the formalism we introduced recovers this extended family of Dixon determinants that already appeared in \cite[Section 5]{Goldman}.
 \end{itemize}
\end{exmp}

\bigskip 

\subsection*{Acknowledgment} We thank the anonymous reviewer for several comments that allowed us to improve the clarity of the presentation, and also for pointing out to us the article \cite{CCD97} (see Remark \ref{rem:ref}).   


\def\cprime{$'$}

\end{document}